\documentclass[12pt]{amsart}

\def\usehyperref{1}

\usepackage{standalone} 
\def\inmain{1}

\usepackage{amssymb, amsmath, amsthm}
\usepackage{calligra, mathrsfs}
\usepackage{upgreek}
\usepackage{graphicx}
\usepackage{url}
\usepackage{mathtools}
\usepackage{enumerate}
\usepackage{verbatim}
\usepackage[retainorgcmds]{IEEEtrantools}
\usepackage{tikz-cd}
\usetikzlibrary{positioning}
\usepackage{microtype}

\usepackage[margin=1in,marginparwidth=0.8in, marginparsep=0.1in]{geometry}

\ifx\usehyperref\undefined
\usepackage{xr}
\newcommand\texorpdfstring[2]{#1}
\newcommand\nolinkurl[1]{\url{#1}}
\else
\usepackage{xr-hyper}
\usepackage[linktocpage=true, unicode]{hyperref}
\usepackage{xcolor}
\hypersetup{
    colorlinks,
    linkcolor={red!65!black},
    citecolor={blue!78!black}
    }
\fi

\newcommand{\Ccal}{{\mathcal C}}
\newcommand{\Dcal}{{\mathcal D}}
\newcommand{\Ecal}{{\mathcal E}}
\newcommand{\Fcal}{{\mathcal F}}
\newcommand{\Gcal}{{\mathcal G}}
\newcommand{\Hcal}{{\mathcal H}}
\newcommand{\Ical}{{\mathcal I}}

\newcommand{\Mcal}{{\mathcal M}}

\newcommand{\Ocal}{{\mathcal O}}

\newcommand{\Scal}{{\mathcal S}}

\newcommand{\Vcal}{{\mathcal V}}

\newcommand{\ccal}{\Ccal}
\newcommand{\dcal}{\Dcal}
\newcommand{\ecal}{\Ecal}

\renewcommand{\SS}{{\mathbb S}}
\newcommand{\ZZ}{{\mathbb Z}}
\newcommand{\GG}{{\mathbb G}}
\newcommand{\QQ}{{\mathbb Q}}

\usepackage{scalerel,stackengine}

\makeatletter
\g@addto@macro\bfseries{\boldmath}
\makeatother

\newcommand{\kr}{\kern -2pt}

\DeclareMathOperator{\id}{id}
\DeclareMathOperator{\Hom}{Hom}

\newcommand{\red}{{\text{\normalfont red}}}
\newcommand{\enh}{{\text{\normalfont enh}}}
\newcommand{\op}{{\text{\normalfont op}}}

\newcommand{\co}{{\normalfont\text{co}}}

\newcommand{\cl}{{\mathrm{cl}}}

\newcommand{\fp}{{\mathrm{fp}}}

\DeclareMathOperator{\Set}{Set}
\DeclareMathOperator{\Cat}{Cat}

\DeclareMathOperator{\colim}{colim}
\DeclareMathOperator{\Tot}{Tot}

\DeclareMathOperator{\Ind}{Ind}

\let\Pr\relax
\DeclareMathOperator{\Pr}{Pr}

\newcommand{\cn}{{\normalfont\text{cn}}}

\DeclareMathOperator{\CAlg}{CAlg}

\DeclareMathOperator{\Mod}{Mod}

\let \Bar \relax
\DeclareMathOperator{\Bar}{Bar}

\newcommand{\D}{\mathrm{D}}

\newcommand{\B}{{\mathrm{B}}}
\newcommand{\GL}{{\mathrm{GL}}}
\newcommand{\bg}{{\B G}}

\DeclareMathOperator{\Aff}{Aff}
\DeclareMathOperator{\Spec}{Spec}

\DeclareMathOperator{\QCoh}{QCoh}
\DeclareMathOperator{\IndCoh}{IndCoh}
\newcommand{\twoQCoh}{{2\kr\QCoh}}
\DeclareMathOperator{\Loc}{Loc}

\newcommand{\st}{{\text{\normalfont st}}}
 
\newtheorem{proposition}[subsubsection]{Proposition}
\newtheorem{lemma}[subsubsection]{Lemma}
\newtheorem{theorem}[subsubsection]{Theorem}
\newtheorem{corollary}[subsubsection]{Corollary}

\newtheoremstyle{note}{8.0pt plus 2.0pt minus 4.0pt}{8.0pt plus 2.0pt minus 4.0pt}{}{}{\bfseries}{.}{.5em}{} 
\theoremstyle{note}
\newtheorem{example}[subsubsection]{Example}
\newtheorem{remark}[subsubsection]{Remark}
\newtheorem{notation}[subsubsection]{Notation}

\newtheorem{warning}[subsubsection]{Warning}
\newtheorem{definition}[subsubsection]{Definition}

\AtBeginDocument{\def\MR#1{}}

\title{Dualizability of derived categories of \linebreak algebraic stacks}

\author{G. Stefanich}

\date{}

\begin{document}


\begin{abstract}
We show that, for a Noetherian algebraic stack with quasi-affine diagonal $X$, the stable $\infty$-category of quasi-coherent sheaves on $X$ is dualizable if and only if the reduced identity component of the stabilizer of $X$ at every geometric point of positive characteristic is a torus. Along the way, we show that this condition on stabilizers is also equivalent to an array of other categorical conditions of interest.
\end{abstract}


\maketitle

\tableofcontents


\section{Introduction}

A fundamental result of Neeman \cite{NeemanGrothendieckDuality} states that the stable $\infty$-category of quasi-coherent sheaves on a quasi-compact separated scheme is compactly generated. Neeman's proof, and its subsequent extension to the quasi-separated context \cite{BondalvandenBergh}, relies at a fundamental level on the localization theorem of Thomason-Trobaugh, and essentially reduces compact generation to a Zariski local assertion.

When working with algebraic stacks, the Zariski local structure is more complex than in the case of schemes, and hence the question of compact generation is much more subtle. As it turns out, in fact, one does not always have compact generation: it was shown by Neeman \cite[remark~4.2]{NeemanHomotopy} that compact generation fails for the classifying stack of the additive group in positive characteristic. 

In general, for an algebraic stack $X$ the presence of a positive characteristic field valued point of $X$ whose stabilizer contains a copy of the additive group as a subgroup is an obstruction to compact generation \cite{HallNeemanRydh}. Following Hall, Neeman and Rydh, we call such stacks poorly stabilized, and use the term well stabilized to refer to algebraic stacks which are not poorly stabilized. We note that well stabilization is a fairly restrictive condition in positive characteristic: every affine algebraic group over an algebraically closed field contains a copy of the additive group unless its reduced identity component is a torus.

This paper is concerned with a weakening of the notion of compact generation known as dualizability. In \cite{HA}, Lurie constructed a symmetric monoidal structure on the $\infty$-category $\Pr_\st$ of presentable stable $\infty$-categories, and it is a basic fact that every compactly generated presentable stable $\infty$-category defines a dualizable object in $\Pr_\st$. The class of dualizable objects of $\Pr_\st$ includes however many $\infty$-categories of interest which are known not to be compactly generated; examples arise frequently in contexts such as microlocal sheaf theory, almost mathematics, condensed mathematics, and the theory of motives.

While more general than compact generation, dualizability allows for many of the same manipulations and constructions that one is used to in the compactly generated case. Notably, Efimov showed that K-theory may be defined for arbitrary dualizable stable $\infty$-categories \cite{Efimov}. And in the setting of algebraic geometry, the categorical K\"{u}nneth formulas in the style of \cite{BZFN} really only require dualizability as opposed to compact generation.

The first goal of this paper is to show that the failure of compact generation for poorly stabilized stacks is a shadow of a much more serious pathology: lack of dualizability. Conversely, we show that, under mild hypotheses, poor stabilization turns out to be the only obstruction to dualizability:

\begin{theorem}\label{theorem main}
Let $X$ be a Noetherian algebraic stack with affine stabilizers and separated diagonal.\footnote{In particular, the theorem applies whenever $X$ is Noetherian and has quasi-affine diagonal.} Then $\QCoh(X)$ is dualizable if and only if $X$ is well stabilized.
\end{theorem}

Theorem \ref{theorem main} may really be thought of as consisting of two separate results, one about dualizability and one about non-dualizability, with relative orthogonal methods of proof, which we summarize next.

The non-dualizability half of the theorem can easily be reduced to the case when $X$ is the classifying stack $\B\mathbb{G}_{a, k}$ of the additive group over a field of positive characteristic $k$. In this case, what we do is construct an endofunctor $F: \QCoh(\B\mathbb{G}_{a, k}) \rightarrow \QCoh(\B\mathbb{G}_{a,k})$ with a colimit preserving right adjoint, such that the colimit in $\Pr_{\st}$ of the sequence
\[
\QCoh(\B\mathbb{G}_{a, k}) \xrightarrow{F} \QCoh(\B\mathbb{G}_{a, k}) \xrightarrow{F} \QCoh(\B\mathbb{G}_{a, k}) \xrightarrow{F} \ldots
\]
is equivalent to the $\infty$-category of $k$-module spectra. General facts about filtered colimits of dualizable $\infty$-categories from \cite{Fully} then imply that, if $\QCoh(\B\mathbb{G}_{a, k})$ were dualizable, it would have nonzero compact objects, which is know to not be the case thanks to work of Hall, Neeman and Rydh \cite{HallNeemanRydh}. Our construction of $F$ makes use of certain symmetries of $\QCoh(\B \mathbb{G}_{a,k})$ which are only present in positive characteristic. 

This argument can be adapted to show that, even if we work in characteristic zero, $\QCoh(\B \mathbb{G}_{a, k}^\infty)$ is not dualizable.\footnote{We are grateful to Dennis Gaitsgory for pointing this out.} More generally, we will show in appendix \ref{appendix infinite type} that dualizability often fails for the classifying stacks of infinite type group schemes; this includes for instance the group $G[[t]]$ of formal paths in an affine algebraic group of positive dimension over a field of characteristic zero (see example \ref{example paths}).

Let us now comment on the proof of the dualizability half of theorem \ref{theorem main}. For algebraic stacks of finite type over a field of characteristic zero (where well stabilization is automatic), dualizability was established in work of Drinfeld and Gaitsgory \cite{DrinfeldGaitsgory}. Their proof depends on the fact that the $\infty$-category $\IndCoh(X)$ of ind-coherent sheaves on $X$ is compactly generated, which they ultimately prove by constructing enough coherent sheaves on $X$. Although, as we shall see below, $\IndCoh(X)$ turns out to be compactly generated in general for well stabilized $X$, Drinfeld and Gaitsgory's argument does not apply away from characteristic zero, since coherent sheaves on $X$ may not necessarily be compact in $\IndCoh(X)$.

What we do instead is prove dualizability by first generalizing a different theorem of Gaitsgory, namely his $1$-affineness for algebraic stacks in characteristic zero \cite{G}. This theorem concerns the $\infty$-category $\twoQCoh(X)$ of sheaves of presentable stable $\infty$-categories on $X$,  whose objects consist roughly speaking  of coherent families of assignments of an $A$-linear presentable stable $\infty$-category to each point $\Spec(A) \rightarrow X$ of $X$. 

In the same way that modules over $\Ocal(X)$ may be localized to give rise to quasi-coherent sheaves on $X$, there is a canonical functor
\[
\Loc_X : \Mod_{\QCoh(X)} \rightarrow \twoQCoh(X) 
\]
where the source is the $\infty$-category of presentable stable $\infty$-categories tensored over $\QCoh(X)$.  In \cite{G}, Gaitsgory defined a stack to be $1$-affine if $\Loc_X$ is an equivalence, and showed that every algebraic stack with affine diagonal and of finite type over a field of characteristic zero is $1$-affine. Our main result on this topic extends his theorem beyond the characteristic zero and affine diagonal setting, and in fact shows that, under mild hypotheses, the class of well stabilized stacks is precisely the class of stacks for which $1$-affineness holds:

\begin{theorem}\label{theorem main affineness}
Let $X$ be a Noetherian algebraic stack with affine stabilizers and separated diagonal. Then $X$ is $1$-affine if and only if $X$ is well stabilized.
\end{theorem}

We remark that theorem \ref{theorem main affineness} and its methods of proof are of a different nature to the $1$-affineness results we proved previously in \cite{Tannaka}. The upshot is that $1$-affineness is sensitive to the kinds of sheaves of categories one works with:  one may get results that apply even in the poorly stabilized context, but it requires one to work with sheaves of stable $\infty$-categories equipped with complete t-structures.

The fundamental fact that connects theorems \ref{theorem main} and  \ref{theorem main affineness}  is proven below as proposition \ref{proposition 1affine implies dualizable}: if $X$ is $1$-affine then $\QCoh(X)$ is dualizable. More generally, we will see that $1$-affineness guarantees dualizability of twisted forms of $\QCoh(X)$. As a consequence, we obtain:

\begin{theorem}\label{theorem main twisted}
Let $X$ be a Noetherian algebraic stack with affine stabilizers and separated diagonal and let $\mathcal{F}$ be a dualizable object of $\twoQCoh(X)$. Then its $\infty$-category of global sections $\Gamma(X, \mathcal{F})$ is dualizable.
\end{theorem}

The above result is a dualizability version of previous theorems that apply to compact generation \cite{ToenAzumaya, DAGXII, HallRydhPerfect}. We note that the known results on compact generation place fairly strict conditions on the diagonal of $X$; in particular, it seems to be unknown whether they hold in the case of the classifying stack of the multiplicative group.

As an application of theorem \ref{theorem main twisted}, we will prove the following variant of theorem \ref{theorem main} which applies to the $\infty$-category $\IndCoh(X)$ of ind-coherent sheaves of $X$:

\begin{theorem}\label{theorem main indcoh}
Let $X$ be a Noetherian algebraic stack with affine stabilizers and separated diagonal. The the following are equivalent:
\begin{enumerate}[\normalfont \indent(1)]
\item $\IndCoh(X)$ is dualizable.
\item $\IndCoh(X)$ is compactly generated.
\item $X$ is well stabilized.
\end{enumerate}
\end{theorem}

We reiterate that, contrary to what the name suggests, $\IndCoh(X)$ is not generally the ind-completion of $\operatorname{Coh}(X)$, even in the well stabilized context. 
 Our proof of compact generation of $\IndCoh(X)$ for well stabilized stacks is in fact relatively indirect, and goes through the theory of dualizable $\infty$-categories. The fundamental ingredient is that an extension of presentable stable $\infty$-categories is dualizable then is automatically compactly generated (\cite[proposition~3.3]{Efimov}). Combined with Noetherian induction this reduces us to proving compact generation on some dense open substack, which is a manageable task. 
  
 We note that this same strategy could be employed to try to prove compact generation in the quasi-coherent setting, but it breaks down at a single point: one needs that compact generation on a closed substack implies compact generation for its formal completion, which is true for $\IndCoh$ but unclear for $\QCoh$.
 
 Our next result concerns a basic point in the homological algebra of quasi-coherent sheaves. Recall that for an algebraic stack $X$ the $\infty$-category $\QCoh(X)$ admits a t-structure, and there is a comparison map $\mathrm{D}(\QCoh(X)^\heartsuit) \rightarrow \QCoh(X)$ from the derived $\infty$-category of the heart. As it turns out, this map is not always an equivalence; this fails to be the case, once again, for the classifying stack of the additive group in positive characteristic \cite{NeemanComplete}, and, more generally, whenever $X$ is poorly stabilized \cite{HallNeemanRydh}. The following theorem provides a converse to this assertion:
 
 \begin{theorem}\label{theorem main derived}
 Let $X$ be a Noetherian algebraic stack with affine diagonal. Then $\QCoh(X)$ is the derived $\infty$-category of its heart if and only if $X$ is well stabilized.
 \end{theorem}
 
 As we shall see, the if direction in theorem \ref{theorem main derived} is deduced from the compact generation part of theorem \ref{theorem main indcoh}, which as we discussed above is a consequence of dualizability and hence it is ultimately dependent on our $1$-affineness result.
 
The difference between $\QCoh(X)$ and the derived $\infty$-category of its heart is a reflection of the potentially ill behavior of Postnikov towers in $\mathrm{D}(\QCoh(X)^\heartsuit)$, and is closely related to another potential pathology: the unbounded cohomological amplitude of products. Our next theorem shows that, once again, well stabilized stacks are precisely the ones for which this problem disappears:

 \begin{theorem}\label{theorem main products}
 Let $X$ be a Noetherian algebraic stack with affine diagonal. Then the following are equivalent:
\begin{enumerate}[\normalfont \indent(1)]
\item Products in $\QCoh(X)$ have bounded cohomological amplitude.
\item Products in $\D(\QCoh(X)^\heartsuit)$ have bounded cohomological amplitude.
\item $X$ is well stabilized.
\end{enumerate} 
 \end{theorem}
 
Under the additional hypothesis that $X$ has finite cohomological dimension, the boundedness of products and the fact that $\QCoh(X)$ is the derived category of its heart were proven previously by Hall \cite{HallRemarks}. This extra assumption places further restrictions on the stabilizer groups away from characteristic zero: for instance, classifying stacks of finite group schemes over a field are well stabilized but do not always have finite cohomological dimension. In fact, Hall's argument rests on the fact that affine covers of $X$ give rise to descendable $E_\infty$-algebras in the sense of \cite{Mathew}, which turns out to be false in general for well stabilized stacks (and essentially only holds when $X$ has finite cohomological dimension, see remark \ref{remark descendability}). From this perspective, we may regard $1$-affineness as a replacement for descendability which has the virtue of applying beyond the finite cohomological dimension setting.
 
 The methods that we use throughout the paper belong to spectral algebraic geometry. We will in fact in the main body of the paper use the term algebraic stack to mean spectral algebraic stack, and we will, by default, prove our results in that context.  The condition of having affine stabilizers is then naturally replaced by the condition of having $0$-affine diagonal. We will study this notion in section \ref{section 1affineness}; for the moment we just mention that  $0$-affineness of the diagonal is implied by having quasi-affine diagonal, and in the case when $\Ocal_X$ is truncated it is equivalent to having affine stabilizers. 
 
 We may summarize our results in the spectral setting as follows:

\begin{theorem}\label{theorem main spectral}
Let $X$ be a Noetherian spectral algebraic stack with $0$-affine and separated diagonal. Then the following are equivalent:
\begin{enumerate}[\normalfont\indent(1)]
\item $\QCoh(X)$ is dualizable.
\item $X$ is $1$-affine.
\item Products in $\QCoh(X)$ have bounded cohomological amplitude.
\item $X$ is well stabilized.
\end{enumerate}
If the above holds then $\IndCoh(X)$ is compactly generated. Conversely, if $\IndCoh(X)$ is dualizable and $\Ocal_X$ is truncated then the above conditions hold.
\end{theorem}

If we restrict to characteristic zero, the dualizability of $\QCoh(X)$ and the $1$-affineness of $X$ were previously only proven under the additional hypothesis that $\Ocal_X$ is truncated, and theorem \ref{theorem main spectral} removes this requirement. 

It is natural to expect that the equivalence between conditions (1) through (4) in theorem \ref{theorem main spectral} holds without the conditions of Noetherianity and separatedness of the diagonal. Although we do not know this in all generality, we will show that the conditions can indeed be relaxed in some cases. If $X$ is a spectral algebraic stack over a field, or a spectral Deligne-Mumford stack, then the equivalence between conditions (1) through (4) holds without any Noetherianity assumptions. Likewise, one can get positive results in the absence of separatedness of the diagonal by imposing stronger conditions on the stabilizer groups. We refer the reader to section \ref{section proofs} for details. 


\subsection{Acknowledgments}

I am grateful to Alexander Efimov, Dennis Gaitsgory, and Peter Scholze for conversations related to the subject of this paper. This work was carried out at the Max Planck Institute for Mathematics in Bonn, and I am grateful to the institute for its hospitality and support.


\subsection{Conventions and notation}

In the main body of this paper we use the convention where the word category refers to $\infty$-category.  We work throughout the paper in the setting of spectral algebraic geometry. The basic objects are called prestacks, and are by definition accessible functors from the category of connective commutative ring spectra into the category of anima. Corepresentable functors are called affine schemes.

We equip the category of affine schemes with the fpppf topology\footnote{Except in the appendix, where we work with infinite type group schemes, which are better dealt with using the fpqc topology.}, whose covers are generated by finite collections of maps $\Spec(B_i) \rightarrow \Spec(A)$ such that the map $A \rightarrow \prod B_i$ is faithfully flat and almost of finite presentation in the sense of \cite{HA} definition 7.2.4.26. Prestacks which satisfy descent with respect to this topology are called stacks.

The class of higher algebraic stacks is then defined inductively as the smallest class of stacks containing affine schemes and closed under small coproducts and geometric realizations of flat almost finitely presented groupoids. An algebraic stack is a higher algebraic stack whose values on classical affine schemes is $1$-truncated, while an algebraic space is a higher algebraic stack whose values on classical affine schemes is $0$-truncated.

An affine cover of an algebraic stack is a jointly surjective family $p_\alpha : U_\alpha \rightarrow X$ of flat almost finitely presented morphisms, where each $U_\alpha$ is affine. We will for the most part work with quasi-compact algebraic stacks; in this case affine covers often consist of a single map $p: U \rightarrow X$. An algebraic stack is said to be classical (resp. Noetherian) if whenever we take such a cover we have that $\Ocal(U_\alpha)$ is $0$-truncated (resp. Noetherian) for every $\alpha$.

Let $\mathbb{S}$ be the sphere spectrum. For each commutative ring spectrum $R$ we denote by $\Mod_R$ the presentable stable category of $R$-modules. Given an affine scheme $S$, we set $\QCoh(S) = \Mod_{\Ocal(S)}$. In general, if $X$ is a prestack we have a corresponding presentable symmetric monoidal category $\QCoh(X)$ defined via Kan extension from the case of affine schemes. If $S$ is a Noetherian affine scheme, we also have a category $\IndCoh(S)$ which is defined as the ind-completion of the full subcategory of $\QCoh(S)$ on the coherent sheaves (that is those sheaves which have finitely many nonzero homologies, which are all finitely generated). There is a star pullback on $\IndCoh$ along flat morphisms, which one then uses to define $\IndCoh(X)$ for more general Noetherian algebraic stacks by descent.

We say that a morphism of connective commutative ring spectra $R \rightarrow R'$ is of finite presentation to order $n$ if $\tau_{\leq n}(R')$ is compact as an object of $\CAlg((\Mod_R)^\cn_{\leq n})$. In the limit $n \to \infty$ this recovers the notion of almost finite presentation; we will otherwise be interested in this notion mainly in the case $n = 0$ where it recovers, for $0$-truncated $R$ and $R'$, the more classical notion of finite presentation from commutative algebra. In general, almost finite presentation is stronger than finite presentation to order $0$, however these notions turn out to be equivalent whenever $R'$ is flat over $R$, a fact which implies that the notion of classical algebraic stack we work with in this paper agrees with the traditional notion defined in terms of finite presentation to order $0$. The notion of finite presentation to order $n$ is extended to morphisms of algebraic stacks in a familiar fashion, by requiring local finite presentation to order $n$ together with quasi-compactness and quasi-separatedness.

For each connective commutative ring spectrum $R$ we denote by $\mathbb{G}_{a, R}$ and $\mathbb{G}_{m, R}$ the (flat) additive and multiplicative group schemes over $R$. To facilitate manipulations with group schemes, we adopt the following convention: if $G$ is  a group over a base $X$ and we have a map $Y \rightarrow X$ we denote by $G_Y$ the pullback of $G$ to $Y$. If $Y = \Spec(R)$, we will sometimes denote this simply by $G_R$.

Let $\Pr$ be the category of presentable categories and colimit preserving functors, and denote by $\Pr_\st$ the full subcategory of $\Pr$ on the presentable stable categories. We equip $\Pr$ and $\Pr_\st$ with their usual symmetric monoidal structures. If $\Ccal_\alpha$ is a diagram of presentable categories and colimit preserving functors we will denote by $\colim \Ccal_\alpha$ their colimit taken in $\Pr$. For each presentable symmetric monoidal stable category $\Mcal$ we denote by $\Mod_\Mcal$ the corresponding category of modules over $\Mcal$ in $\Pr_\st$. If $\Ccal$ is a Grothendieck abelian category we have a corresponding derived category, which is a presentable stable category which we denote $\D(\Ccal)$.


\ifx\inmain\undefined
\bibliographystyle{myamsalpha2}
\bibliography{References}
\fi


\section{Well stabilized algebraic stacks} \label{section well stabilized}

The central theme of this paper is the idea that a lot of the potential pathological behavior of the category of quasi-coherent sheaves on a qcqs algebraic stack with affine stabilizers $X$ can ultimately be traced back to the single example of the classifying stack of the additive group in positive characteristic.  With this in mind, one says that $X$ is well stabilized if for every positive characteristic field valued point $x: \Spec(k) \rightarrow X$ the stabilizer group of $X$ at $x$ does not contain a copy of $\GG_{a, k}$. The goal of this section is to collect some basic geometric results concerning the class of well stabilized stacks.

We begin in \ref{subsection good groups} with a general study of good algebraic groups, which are those algebraic groups whose geometric fibers do not contain copies of positive characteristic additive groups. As we shall see, while this condition is automatic in characteristic zero, it is very restrictive in positive characteristic: it was essentially shown in \cite{HallNeemanRydh} that if $G$ is a good algebraic group over a positive characteristic field $k$ then, after passing to a finite field extension of $k$, the group $G$ contains a split torus whose quotient is finite. More generally, we prove here that an analogous description holds generically for families of good algebraic groups over an integral base (proposition \ref{proposition structure G}).

We then study well stabilized stacks in \ref{subsection well stabilized stacks}. 
 As remarked above, one of the central goals of the paper is to show that well stabilized stacks enjoy many good categorical properties. There is however one desirable property that is not always satisfied by well stabilized stacks: the structure sheaf is not necessarily compact. As explained in \cite{HallRydhGroups}, this issue is resolved by placing stronger conditions on the stabilizer groups. This leads us to the class of very well stabilized stacks, which have the feature that their stabilizer groups at positive characteristic field valued points are linearly reductive. We close this section with a proof of the fact that, under mild conditions, well stabilized stacks admit quasi-finite flat covers by very well stabilized stacks (theorem \ref{theorem quasi-finite structure}).


\subsection{Good algebraic groups}\label{subsection good groups}

We begin with the following general definition:

\begin{definition}\label{definition poor}
Let $S$ be an affine scheme. An algebraic group over $S$ is a flat almost finitely presented group algebraic space $G$ over $S$. We say that $G$ is poor if there exists a field $k$ of positive characteristic and a morphism $\Spec(k) \rightarrow S$ such that $G_{\Spec(k)}$ contains a subgroup isomorphic to $\mathbb{G}_{a, k}$. Otherwise, we say that $G$ is good.
\end{definition}

Definition \ref{definition poor} is a version in families of the notion of poor algebraic group from \cite{HallNeemanRydh} section 4, and recovers it in the case when $S$ is the spectrum of an algebraically closed field:

\begin{proposition}\label{proposition good over algebraically closed}
Let $k$ be an algebraically closed field and let $G$ be an algebraic group over $k$. The following are equivalent:
\begin{enumerate}[\normalfont \indent (1)]
\item The reduction of the identity component of $G$ is an extension of an abelian variety by a torus.
\item There is no subgroup of $G$ isomorphic to $\mathbb{G}_{a, k}$.
\item For every field extension $F$ of $k$, there is no subgroup of $G_F$ isomorphic to $\mathbb{G}_{a, F}$.
\end{enumerate}
\end{proposition}

The equivalence between (1) and (2) in proposition \ref{proposition good over algebraically closed} features already in \cite{HallNeemanRydh}, so our only new claim is that (2) implies (3). The proof of this will need some preliminaries.

\begin{notation}
Let $S$ be a classical affine scheme. We denote by $\operatorname{AlgStk}^{\cl, \fp}_S$ the category of classical algebraic stacks of finite presentation to order $0$ over $S$, and by $\operatorname{Grp}(\operatorname{AlgStk}^{\cl,\fp}_S)$ the category of group objects in $\operatorname{AlgStk}^{\cl,\fp}_S$. Note that for each morphism of classical affine schemes $S \rightarrow T$ we have functors
\[
\operatorname{AlgStk}^{\cl,\fp}_T \rightarrow \operatorname{AlgStk}^{\cl,\fp}_S
\]
and
\[
\operatorname{Grp}(\operatorname{AlgStk}^{\cl,\fp}_T) \rightarrow \operatorname{Grp}(\operatorname{AlgStk}^{\cl,\fp}_S),
\]
induced by pullback. 
\end{notation}

\begin{lemma}\label{lemma colimit}
Let $S_\alpha$ be a cofiltered system of classical affine schemes, and let $S = \lim S_\alpha$. Then the canonical functors
\[
\colim \operatorname{AlgStk}^{\cl,\fp}_{S_\alpha} \rightarrow \operatorname{AlgStk}^{\cl,\fp}_S
\]
and 
\[
\colim \operatorname{Grp}(\operatorname{AlgStk}^{\cl,\fp}_{S_\alpha}) \rightarrow \operatorname{Grp}(\operatorname{AlgStk}^{\cl,\fp}_S)
\]
are equivalences.
\end{lemma}
\begin{proof}
The first part of the lemma follows directly from \cite{LaumonMoretBailly} proposition 4.18. For the second part of the lemma we note that if $\ccal_\alpha$ is a filtered system of $(1,1)$-categories with finite products and transitions which preserve finite products, then the canonical functor $\colim \operatorname{Grp}(\ccal_\alpha) \rightarrow \operatorname{Grp}(\colim \ccal_\alpha)$ is an equivalence.
\end{proof}

\begin{lemma}\label{lemma embeddings}
Let $S$ be a classical affine scheme, and let $H, G$ be a pair of algebraic groups over $S$. Consider the functor
\[
\operatorname{Emb}_{H, G} : (\Aff^\cl_{/S})^\op \rightarrow \Set
\]
which sends each classical affine scheme $T$ over $S$ to the set of group embeddings $H_T \rightarrow G_T$. Then $\operatorname{Emb}_{H, G}$ preserves filtered colimits.
\end{lemma}
\begin{proof}
Follows from lemma \ref{lemma colimit}, since embeddings $H_T \rightarrow G_T$ are the same as monomorphisms from $H_T$ to $G_T$ in $\operatorname{Grp}(\operatorname{AlgStk}^{\normalfont\text{fp}}_T)$.
\end{proof}

\begin{proof}[Proof of proposition \ref{proposition good over algebraically closed}]
The equivalence between (1) and (2) is \cite{HallNeemanRydh} lemma 4.1. The fact that (3) implies (2) is clear. We finish by showing that (2) implies (3). Assume given a field extension $F$ of $k$ such that $G_F$ contains a copy of $\mathbb{G}_{a, F}$. We need to show that $G$ contains a copy of $\mathbb{G}_{a, k}$. Write $F$ as a filtered colimit of classical finitely generated $k$-algebras $A_\alpha$. Applying lemma \ref{lemma embeddings} we deduce that there exists an index $\alpha$ such that $G_{A_\alpha}$ contains a copy of $\mathbb{G}_{a, A_\alpha}$. Since $k$ is algebraically closed we may choose a retraction $A_\alpha \rightarrow k$. Changing base along this morphism yields the desired embedding $\mathbb{G}_{a, k} \rightarrow G$.
\end{proof}

In this paper we will be interested in good \emph{affine} algebraic groups. Proposition \ref{proposition good over algebraically closed} implies the following useful characterization in the case when the base is an algebraically closed field $k$: either the characteristic of $k$ is zero, or there is an inclusion of a split torus in $G$ with finite quotient. If  $k$ is not algebraically closed a similar characterization holds after passing to a finite field extension.  The following proposition shows that such a description also holds generically over any integral base:

\begin{proposition}\label{proposition structure G}
Let $S$ be an integral affine scheme and let $G$ be an algebraic group over $S$ with affine fibers. If $G$ is good then one of the following two conditions hold:
\begin{enumerate}[\normalfont \indent (a)]
\item There exists a dense open subscheme $U$ of $S$ which lies over $\mathbb{Q}$, an integer $n \geq 0$, and a subgroup inclusion $G \rightarrow \GL_{n, U}$ whose quotient is quasi-affine.
\item There exists a dense open subscheme $U$ of $S$, a finite flat\footnote{Finite flat is taken here to mean that the relative structure sheaf is a vector bundle, see definition \ref{definition finite flat} below. In the Noetherian setting, this agrees with the condition of being finite and flat.} surjection $U' \rightarrow U$, an integer $n \geq 0$, and a subgroup inclusion $\GG^n_{m, U'} \rightarrow G_{U'}$  with finite flat quotient.
\end{enumerate}
\end{proposition}

\begin{proof}[Proof of proposition \ref{proposition structure G}]
Let $K$ be the fraction field of $S$, and fix an algebraic closure $\overline{K}$ of $K$. We will show that conditions (a) or (b) hold depending on whether $G_{\overline{K}}$ contains a copy of $\mathbb{G}_{a, \overline{K}}$ or not.

Assume first that there is a subgroup inclusion  $\mathbb{G}_{a, \overline{K}} \rightarrow G_{\overline{K}}$. Applying lemma \ref{lemma embeddings} we may find a finite field extension $F$ of $K$ such that $G_F$ contains a copy of $\mathbb{G}_{a, F}$. Passing to an affine open subset of $S$ if necessary we may find a finite flat surjection $S' \rightarrow S$ such that $\Spec(K) \times_S S' = \Spec(F)$. By \cite{Totaro} lemma 3.1 there exists a nonnegative integer $n$ and a subgroup inclusion $i_F: G_F \rightarrow \GL_{n, F}$ whose quotient is a  quasi-affine scheme. Another application of lemma \ref{lemma embeddings} implies that there exists a dense open affine subscheme $U$ of $S$ such that, setting $U' = S' \times_S U$, the group $G_{U'}$ contains a copy of $\mathbb{G}_{a, U'}$ and admits an embedding $i_{U'}: G_{U'} \rightarrow \GL_{n, U'}$ whose base change to $\Spec(F)$ recovers $i_F$. The quotient $\GL_{n, U'}/G_{U'}$ is a qcqs algebraic space whose base change along $\Spec(K) \rightarrow U$ is quasi-affine; applying \cite{RydhApprox1} proposition B.3 we may, after shrinking $U$ if necessary, assume that $i_{U'}$ has quasi-affine quotient.

It remains to show that $U$ lies over $\QQ$.  Being a base change of $G$, the group $G_{U'}$ is good. The fact that it contains a copy of the additive group then implies that all the residue fields of $U'$ are of characteristic zero. It follows that all the residue fields of $U$ are of characteristic zero, and the desired assertion follows.

Assume now that there is no subgroup inclusion $\mathbb{G}_{a, \overline{K}} \rightarrow G_{\overline{K}}$. By proposition \ref{proposition good over algebraically closed} we have that the reduction of the identity component of $G_{\overline{K}}$ is a torus. In other words, there is an integer $n$ and a subgroup inclusion $i_{\overline{K}}: \mathbb{G}_{m, \overline{K}}^n \rightarrow G_{\overline{K}}$ whose quotient is finite.

Applying lemma \ref{lemma embeddings} we may find a finite field extension $F$ of $K$ inside $\overline{K}$ and an inclusion $i_{F} : \mathbb{G}_{m, F}^n \rightarrow G_{F}$ whose base change to $\overline{K}$ recovers $i_{\overline{K}}$. Since finite schemes descend along $F \rightarrow \overline{K}$, the quotient of $i_F$ is finite. Passing to an affine open subset of $S$ if necessary we may find a finite flat surjection $S' \rightarrow S$ such that $\Spec(K) \times_S S' = \Spec(F)$. Another application of lemma \ref{lemma embeddings} implies that there exists a dense open affine subscheme $V$ of $S$ such that, setting $V' = S' \times_S V$, there exists a subgroup inclusion $i_{V'}: \mathbb{G}_{m, V'}^n \rightarrow G_{V'}$  whose base change to $\Spec(F)$ recovers $i_F$. The quotient of $i_{V'}$ is a classical algebraic space of finite presentation to order $0$ over $V'$ whose base change to $\Spec(F)$ is finite. It now follows from lemma \ref{lemma colimit} that we may find a dense open subscheme $U$ of $V$ such that the base change of $i_{V'}$ to $U' = V' \times_V U$ has finite flat quotient.
\end{proof}
 
It follows from proposition \ref{proposition structure G} that if $G$ is a good algebraic group with affine fibers over a classical Noetherian scheme $S$, then there exists a finite increasing sequence of open subschemes
\[
\emptyset = U_0 \subseteq U_1 \subseteq \ldots \subseteq U_n = S
\]
such that for every $i$ the base change of $G$ to the reduction of $U_{i+1} \backslash U_i$ satisfies one of conditions (a) or (b). One may wonder whether in fact more is true, and locally on $S$ one of conditions (a) or (b) are verified. This turns out to not be the case:

\begin{example}
Let $G_0 = \Spec(\ZZ[x, y][1/(xy + 1)])$, thought of as an algebraic group over $\Spec(\ZZ[x])$, where the structure map is given by
\[
y \mapsto x (y \otimes 1) (1 \otimes y) + ( y \otimes 1 ) + (1 \otimes y).
\]
Here the base change of $G_0$ to the locus $x = 0$ recovers the group $\GG_{a, \mathbb{Z}}$, while its base change to the locus $x \neq 0$ recovers the group $\GG_{m, \mathbb{Z}[x, x^{-1}]}$.

Let $S \subseteq \Spec(\ZZ[X])$ be the intersection of the locus $x - p \neq 0$ over all prime numbers $p$, and let $G$ be the base change of $G_0$ to $S$. Then $G$ is a good algebraic group over $S$. The locus where $x = 0$ on $S$ consists now of a single closed point $x$, and there does not exist an open neighborhood of $x$ satisfying either of the conditions from proposition \ref{proposition structure G}.
\end{example}


\subsection{Well stabilized and very well stabilized stacks}\label{subsection well stabilized stacks}

For every field valued point $x: \Spec(k) \rightarrow X$ in an algebraic stack one may define the (classical) stabilizer group of $X$ at $x$, which is the group algebraic space $k$ given by the classical truncation of $\Spec(k) \times_{X \times X} X$. If $X$ is quasi-separated, then its stabilizers are in fact algebraic groups.   Equipped with this notion, we arrive at the main definition of this section:

\begin{definition}\label{definition well stabilized}
Let $X$ be a quasi-separated algebraic stack. We say that $X$ is well stabilized if the stabilizer group of $X$ at every field valued point is good. Otherwise we say that $X$ is poorly stabilized.
\end{definition}

\begin{remark}
Let $k \rightarrow K$ be a field extension, and $G$ be an algebraic group over $k$. Then $G$ is good if and only if $G_K$ is good. It follows that if $x: \Spec(k) \rightarrow X$ is a field valued point of a quasi-separated algebraic stack $X$, then the condition that the stabilizer of $X$ at $x$ be good only depends on the induced point in the topological space underlying $X$. In particular, in definition \ref{definition well stabilized} one may restrict attention to stabilizer groups at points valued in algebraically closed fields. This is the form in which definition \ref{definition well stabilized} is formulated in \cite{HallNeemanRydh} section 4.
\end{remark}

Definition \ref{definition well stabilized} specializes to definition \ref{definition poor} in the case of classifying stacks:

\begin{proposition}\label{proposition BG well stabilized iff G good}
Let $S$ be an affine scheme and let $G$ be an algebraic group over $S$. Then $\B G$ is well stabilized if and only if $G$ is good. 
\end{proposition}
\begin{proof}
Assume that $G$ is good, and suppose given a field valued point $x: \Spec(k) \rightarrow \B G$. Let $K$ be a field extension of $k$ such that the composite map $x': \Spec(K) \rightarrow \Spec(k) \rightarrow \B G$ factors through $S$. To show that the stabilizer of $\B G$ at $x$ is good it suffices to show that the stabilizer at $x'$ is good. Indeed, this agrees with $G_K$, which is good since $G$ is good.

Assume now that $\B G$ is well stabilized, and suppose given a field $k$ and a morphism $\Spec(k) \rightarrow S$. Then $G_k$ is the stabilizer of $\B G$ at the composite map $\Spec(k) \rightarrow S \rightarrow \B G$, and consequently it is good. It follows that $G$ is good, as desired.
\end{proof}

We also have the following basic observation:

\begin{proposition}\label{proposition algebraic space over well stabilized}
Let $p: X \rightarrow Y$ be a morphism of quasi-separated algebraic stacks. Assume that $Y$ is well stabilized and that $p$ is representable by algebraic spaces. Then $X$ is well stabilized.
\end{proposition}
\begin{proof}
Let $x: \Spec(k) \rightarrow X$ be a field valued point. Then the stabilizer of $X$ at $x$ embeds inside the stabilizer of $Y$ at $p \circ x$. The result now follows from the fact that a subgroup of a good algebraic group is good.
\end{proof}

We now introduce the particularly nicely-behaved class of very well stabilized stacks.

\begin{definition}
Let $G$ be an algebraic group over an affine scheme $S$, with affine fibers. We say that $G$ is very good if for every algebraically closed field $k$ of positive characteristic and every map $\Spec(k) \rightarrow S$ we have that the connected component of the identity in $G_{\Spec(k)}$ is of multiplicative type and the number of connected components of $G_{\Spec(k)}$ is not divisible by $p$. We say that a quasi-separated algebraic stack $X$ with affine stabilizers is very well stabilized if the stabilizer group of $X$ at every field valued point is very good.
\end{definition}

\begin{remark}\label{remark equivalences very good}
Let $X$ be a Noetherian algebraic stack with affine stabilizers. It was shown in \cite{HallRydhGroups} section 2 that if $X$ is classical then the following conditions are equivalent:
\begin{enumerate}[\normalfont \indent(1)]
\item $\Ocal_X$ is a compact object of $\QCoh(X)$.
\item The functor $\Gamma(X, -): \QCoh(X) \rightarrow \Mod_\SS$ has bounded cohomological amplitude (see definition \ref{definition bounded cohomological amplitude}).
\item $X$ is very well stabilized.
\end{enumerate}
Although we will only need to make use of this in the classical setting, we mention that this equivalence continues to hold when $X$ is possibly spectral. In this case, one shows that conditions (1) and (2) are equivalent, and that they hold for $X$ if and only if they hold for $X_\cl$, see for instance the arguments in \cite{DrinfeldGaitsgory} sections 2.1 and 2.2.
\end{remark}

The following result is a variant of proposition \ref{proposition BG well stabilized iff G good}, and is proven in the same way:

\begin{proposition}
Let $S$ be an affine scheme and let $G$ be an algebraic group over $S$ with affine fibers. Then $\B G$ is very well stabilized if and only if $G$ is very good. 
\end{proposition}

We also have the following variant of proposition \ref{proposition algebraic space over well stabilized}:

\begin{proposition}
Let $p: X \rightarrow Y$ be a morphism of quasi-separated algebraic stacks with affine stabilizers. Assume that $Y$ is very well stabilized and that $p$ is representable by algebraic spaces. Then $X$ is very well stabilized.
\end{proposition}
\begin{proof}
As in the proof of proposition \ref{proposition algebraic space over well stabilized}, the central observation is that if $G$ is a very good algebraic group over a field $k$ and $H$ is a subgroup of $G$, then $G$ is very good. This follows for instance by using characterization (1) in remark \ref{remark equivalences very good}: the structure sheaf of $\B H$ is compact since it is the pullback of $\Ocal_{\B G}$ along the map $q: \B H \rightarrow \B G$, and $q_*$ preserves colimits. 
\end{proof}

The remainder of this section is devoted to the proof of the following result, which shows that well stabilized algebraic stacks often admit quasi-finite flat covers by very well stabilized algebraic stacks.

\begin{theorem}\label{theorem quasi-finite structure}
Let $X$ be a classical Noetherian algebraic stack with affine stabilizers. Assume that $X$ is well stabilized and has separated diagonal. Then there exists a morphism $f: X' \rightarrow X$ which is flat, surjective, quasi-finite, separated, representable by algebraic spaces, and such that $X'$ is very well stabilized.
\end{theorem}

\begin{lemma}\label{lemma Maschke}
Let $k$ be an algebraically closed field, and let $G$ be an affine algebraic group over $k$. If the reduced identity component $G^0_\red$ is a torus and the dimension of $\Ocal(G/G^0_\red)$ is invertible in $k$ then $G$ is very good.
\end{lemma}
\begin{proof}
Let $\Fcal$ be the pushforward of $\Ocal_{\B G^0_\red}$ along the map $p: \B G^0_\red \rightarrow \B G$. Then $\Fcal$ is a vector bundle over $\B G$, and the composition of the unit map $\Ocal_{\B G} \rightarrow \Fcal \otimes \Fcal^\vee$ and the counit map $\Fcal \otimes \Fcal^\vee \rightarrow \Ocal_{\B G}$ is given by multiplication by the dimension of $\Ocal(G/G^0_\red)$. It follows that $\Ocal_{\B G}$ is a retract of $\Fcal \otimes \Fcal^\vee = p_*p^*(\Fcal^\vee)$. Since $G_\red$ is a torus we have that $p^*\Fcal^\vee$ is a projective object of $\QCoh(\B G^0_\red)^\heartsuit$, and since $p_*$ admits a t-exact colimit preserving right adjoint we deduce that $\Fcal \otimes \Fcal^\vee$ is a projective object of $\QCoh(\B G)^\heartsuit$. It now follows that $\Ocal_{\B G}$ is projective as well, which implies by \cite{HallRydhGroups} theorem 1.2 that $G$ is very good.
\end{proof}

\begin{lemma}\label{lemma very good after inverting}
Let $S$ be a classical Noetherian affine scheme and let $G$ be an algebraic group over $S$. Assume that $G$ is good and has affine fibers. Then there exists a dense open subscheme $U$ of $S$ and a  nonzero integer $N$ such that $G_{U \times_{\Spec{\ZZ}} \Spec(\ZZ[1/N])}$ is very good.
\end{lemma}
\begin{proof}
The notions of good and very good algebraic group are insensitive to non-reduced structures, so we may, by replacing $S$ with $S_\red$, assume that $S$ is reduced. It suffices furthermore to prove the lemma for each irreducible component of $S$, so we may assume that $S$ is integral. 

We apply proposition \ref{proposition structure G}. If condition (a) holds then $G_U$ is automatically very good, so we may take $N = 1$. Assume now that condition (b) holds, so that we have an inclusion $\GG^n_{m, U'} \rightarrow G_{U'}$ whose quotient is finite flat.  For each point $x$ in $U'$ let $N_x$ be the dimension of the vector space of functions on the fiber of $G_{U'}/\GG^n_{m, U'}$ at $x$. Note that $N_x$ is locally constant in $U'$, so we may pick a nonzero integer $N$ which is a multiple of $N_x$ for all $x$. We claim that $G_{U \times_{\Spec{\ZZ}} \Spec(\ZZ[1/N])}$ is very good. Since the map $U' \rightarrow U$ is surjective, it is enough to show that $G_{U' \times_{\Spec{\ZZ}} \Spec(\ZZ[1/N])}$ is very good. This follows from lemma \ref{lemma Maschke}.
\end{proof}

\begin{lemma}\label{lemma very well stabilized after inverting}
Let $X$ be a classical Noetherian algebraic stack with affine stabilizers. Assume that $X$ is well stabilized. Then there exists a nonzero integer $N$ such that $X \times_{\Spec(\ZZ)} \Spec(\ZZ[1/N])$ is very well stabilized.
\end{lemma}
\begin{proof}
The condition of (very) well stabilization only depends on the reduced stack underlying a given stack, so we may by replacing $X$ with $X_\red$ assume that $X$ is reduced. We argue by Noetherian induction. The case when $X$ is empty is clear, so assume that $X$ is nonempty, and that the lemma is known for every proper reduced closed substack of $X$.

Using \cite{stacks-project} proposition 06RC we may find a dense open substack $Y$ of $X$ which is a gerbe over an algebraic space $S$. Let $S' \rightarrow S$ be an affine cover such that $Y$ has a section over $S'$. Then we may write $Y' \times_{S} S' = \B G$ for some algebraic group $G$ over $S'$ with affine fibers.  Applying proposition \ref{proposition algebraic space over well stabilized} we see that $\B G$ is well stabilized, and hence $G$ is good (proposition \ref{proposition BG well stabilized iff G good}). By lemma \ref{lemma very good after inverting} there exists a dense open subscheme $U'$ of $S'$ and a nonzero integer $N_Y$ with the property that $G_{U' \times_{\Spec(\ZZ)} \Spec(\ZZ[1/N_Y])}$ is very good. The image of $U'$ in $S$ then provides a dense open substack $U$ of $S$ with the property that $Y \times_S U \times_{\Spec(\ZZ)} \Spec(\ZZ[1/N_Y])$ is very well stabilized. Replacing $Y$ with $Y \times_S U$ we may now assume that $Y \times_{\Spec(\ZZ)} \Spec(\ZZ[1/N_Y])$ is very well stabilized. 

Let $Z$ be the reduced complement of $Y$. Then our inductive hypothesis implies that there exists a nonzero integer $N_Z$ with the property that $Z \times_{\Spec(\ZZ)} \Spec(\ZZ[1/N_Z])$ is very well stabilized. The lemma now follows by setting $N = N_Z N_Y$.
\end{proof}

\begin{lemma}\label{lemma cover at closed point}
Let $X$ be a classical Noetherian algebraic stack with affine stabilizers and separated diagonal. Let $x$ be a closed point of $X$ of positive characteristic and assume that the stabilizer of $X$ at $x$ is good. Then there exists a morphism $f: W \rightarrow X$ which is flat, quasi-finite, separated, and representable by algebraic spaces, with $x$ contained in the image of $f$, and such that $W$ is very well stabilized.
\end{lemma}
\begin{proof}
Let $Z$ be the residual gerbe of $X$ at $x$, so that $Z$ is a reduced closed substack of $X$ which only contains the point $x$ (\cite{stacks-project} lemma 0509). By \cite{stacks-project} lemma 06QK we have that $Z$ is a gerbe over a reduced algebraic space with a single point. Applying \cite{stacks-project} proposition 06NH we deduce that $Z$ is in fact a gerbe over the spectrum of a field $k$. Let $K$ be a finite field extension of $k$ with the property that $Z$ has a section over $\Spec(K)$. Then $Z \times_{\Spec(k)} \Spec(K) = \B G$ for some algebraic group $G$ over $K$. We have that $G$ is the stabilizer of $X$ at $\Spec(K)$ and hence it is affine and good. After replacing $K$ with a further finite field extension we may assume that there is a subgroup inclusion $\GG^n_{m, K} \rightarrow G$ whose quotient is finite. Let $W_0 = \B \GG^n_{m, K}$, and observe that the composite $W_0 \rightarrow \B G \rightarrow Z$ is finite flat. We claim that this map is furthermore syntomic. To see this, note that the projection $W_0 \rightarrow \B G$ is syntomic since its fiber $G/\GG^n_{m, K}$ is syntomic by \cite{AlperHallHalpernLeistnerRydh} lemma 5.2, while the projection $\B G \rightarrow Z$ is syntomic since it is a base change of the map $\Spec(K) \rightarrow \Spec(k)$.

An application of \cite{AlperHallHalpernLeistnerRydh} theorem 5.1(b) provides a syntomic morphism $f: \mathcal{W} \rightarrow X$ whose base change to $Z$ recovers $W_0$, and such that $\mathcal{W}$ is fundamental (that is, it admits an affine morphism to the classifying stack of a general linear group over the integers). Applying \cite{AlperHallRydhstructure} proposition 6.8 we may, after passing to an \'etale neighborhood of $W_0$, assume that $\mathcal{W}$ is very well stabilized. An application of \cite{AlperHallRydhstructure} proposition 5.3 now shows that, after shrinking $\mathcal{W}$ if necessary, we may arrange that $f$ be separated and representable by algebraic spaces. We now finish by letting $W$ be the locus where $f$ is quasi-finite (which is an open substack of $\mathcal{W}$ containing $W_0$).
\end{proof}

\begin{proof}[Proof of theorem \ref{theorem quasi-finite structure}]
Let $N$ be as in lemma \ref{lemma very well stabilized after inverting}, and for each closed point of positive characteristic $x$ in $X$ let $f_x: W_x \rightarrow X$ be as in lemma \ref{lemma cover at closed point}. By quasi-compactness of $X$ we may pick a finite set of closed points $x_i$ such that the maps $f_{x_i}$ together with the inclusion $X \times_{\Spec(\ZZ)} \Spec(\ZZ[1/N]) \rightarrow X$ cover $X$. We now finish the proof by setting $X'$ to be the disjoint union of $\coprod W_{x_i}$ and $X \times_{\Spec(\ZZ)} \Spec(\ZZ[1/N]) $.
\end{proof}


\ifx\inmain\undefined
\bibliographystyle{myamsalpha2}
\bibliography{References}
\fi


\section{Sheaves of categories, \texorpdfstring{$1$}{1}-affineness and codescent} \label{section 1affineness}

We are in this paper faced with the task of showing that a qcqs algebraic stack $X$ enjoys good categorical properties, from information about the stabilizer groups of $X$. There is a basic roadmap for accomplishing something of this nature in the Noetherian setting:
\begin{enumerate}[\normalfont \indent(1)]
\item One deduces the desired assertion when $X$ is the classifying stack of an algebraic group over a classical base: here one can make direct contact with stabilizer assumptions.
\item Point (1) is then extended to the case when $X$ is a classical gerbe.
\item The case when $X$ is classical is then handled by using stratifications by gerbes.
\item Finally, one deduces the case of possibly spectral $X$ by reducing to the case of $X_{\cl}$.
\end{enumerate}

The basic categorical property we are interested in is dualizability of $\QCoh(X)$. It turns out however that dualizability is somewhat too weak of a condition and does not  immediately lend itself to performing the manipulations needed to carry out the above plan. 

To get around this issue, we will work in this paper with stronger conditions which are amenable to carrying out the above strategy, and which themselves imply dualizability. These conditions will be condensed in section \ref{section controlled} in the notion of a controlled algebraic stack. As we shall see, the property of being controlled is robust enough to allow for most of the manipulations needed in the above roadmap, and implies dualizability as well as a number of other pleasant properties.

The condition of being controlled is closely related to the notion of $1$-affineness introduced by Gaitsgory \cite{G}, which is perhaps the most fundamental concept in the theory of sheaves of categories. Several of the pleasant categorical properties of controlled algebraic stacks, in particular dualizability, are in fact a consequence of the fact that controlled algebraic stacks are $1$-affine. 

The goal of this section is to make a general study of $1$-affineness. We begin in  \ref{subsection dualizability and 1 aff} with general background concerning the theory of sheaves of categories and the notion of $1$-affineness. We prove here the fundamental fact that $1$-affineness of $X$ guarantees dualizability of $\QCoh(X)$ (proposition \ref{proposition 1affine implies dualizable})  and, more generally, dualizability of twisted forms of $\QCoh(X)$ (corollaries \ref{corollary twisted 1} and \ref{corollary twisted 2}).

The notion of $1$-affineness is concerned with an equivalence between the category $\Mod_{\QCoh(X)}$ of module categories over $\QCoh(X)$ and the category $\twoQCoh(X)$ of sheaves of categories on $X$. In \ref{subsection 0aff} we study the condition of $0$-affineness of the diagonal, which guarantees that the global section functors $\twoQCoh(X) \rightarrow \Mod_{\QCoh(X)}$ is fully faithful (proposition \ref{proposition gammaenh fully faithful}). Although $0$-affineness of the diagonal is defined in categorical terms, it admits a pleasant geometric characterization when $\Ocal_X$ is truncated: it is equivalent to the condition that the stabilizers of $X$ be affine (proposition \ref{proposition characterize 0affine}).

Finally, in \ref{subsection Bar} we show that, in the presence of $0$-affine diagonal, $1$-affineness is equivalent to a codescent property for quasi-coherent sheaves: $\QCoh(X)$ is the \emph{colimit} of $\QCoh(S)$ over all affine schemes $S$ equipped with a map to $X$, with transitions given by pushforward functors. This condition is eventually turned into a concrete criterion for checking $1$-affineness in proposition \ref{proposition check affineness via Bar}, which will ultimately form the basis of our proof that controlled algebraic stacks are $1$-affine.


\subsection{Dualizability and \texorpdfstring{$1$}{1}-affineness} \label{subsection dualizability and 1 aff}

We begin with some generalities concerning the theory of sheaves of categories.

\begin{notation}
For each affine scheme $S$ we set
\[
\twoQCoh(S) = \Mod_{\QCoh(S)}.
\]
Given a morphism $f: S \rightarrow T$ between affine schemes, we have a functor
\[
f^*: \twoQCoh(T) \rightarrow \twoQCoh(S)
\]
obtained by extension of scalars along the symmetric monoidal functor $f^* : \QCoh(T) \rightarrow \QCoh(S)$. In this way, the assignment $S \mapsto \twoQCoh(S)$ gives rise to a functor from the opposite of the category of affine schemes to the category of large categories.

The above definition extends to prestacks via right Kan extension. Concretely, for each prestack $X$ we set
\[
\twoQCoh(X) = \lim \Mod_{\QCoh(S)}
\]
where the limit is taken over the category of affine schemes $S$ equipped with a map to $X$. Note that $\twoQCoh(X)$ admits a structure of symmetric monoidal category induced from the canonical symmetric monoidal structures on $\Mod_{\QCoh(S)}$.
\end{notation}

Objects in $\twoQCoh(X)$ are called sheaves of (presentable stable) categories on $X$. This notion has good descent properties:

\begin{theorem}[\cite{G} theorem 1.5.7, \cite{SAG} theorem D.3.6.2\footnote{In light of \cite{SAG} proposition D.3.3.1. This descent result is also essentially contained in \cite{Mathew} proposition 3.45.}]\label{theorem descent twoqcoh}
The assignment $S \mapsto \twoQCoh(S)$ is a sheaf for the fpppf topology.
\end{theorem}

In addition to $\twoQCoh$, the assignment $S \mapsto \Mod_{\QCoh(S)}$ admits another extension to prestacks, namely the functor that sends each prestack $X$ to $\Mod_{\QCoh(X)}$. By the universal property of right Kan extensions we obtain for each $X$ a comparison functor
\[
\Loc_X: \Mod_{\QCoh(X)} \rightarrow \twoQCoh(X).
\]

\begin{definition}\label{definition 1 affineness}
We say that a prestack $X$ is $1$-affine if the functor $\Loc_X$ is an equivalence.
\end{definition}

\begin{example}\label{example alg space}
Let $X$ be a quasi-compact quasi-separated algebraic space. Then $X$ is $1$-affine, see  \cite{G} theorem 2.1.1 or  \cite{SAG} theorem 10.2.0.2.
\end{example}

The following result explains the relevance of definition \ref{definition 1 affineness} to questions of dualizability:

\begin{proposition}\label{proposition 1affine implies dualizable}
Let $X$ be a prestack. If $\Loc_X$ is fully faithful then $\QCoh(X)$ is dualizable.
\end{proposition}
\begin{proof}
We will show that the functor
\[
- \otimes \QCoh(X) : \Pr_{\st} \rightarrow \Pr_{\st}
\]
admits a $\Pr_{\st}$-linear left adjoint. Since $\Pr_{\st}$ is generated under colimit by dualizable objects, $\Pr_{\st}$-linearity of left adjoints is automatic. Hence it suffices to show that for every presentable stable category $\ccal$ the copresheaf 
\[
\Hom_{\Pr_{\st}}(\ccal, - \otimes \QCoh(X))
\]
 is corepresented.

Write $X$ as the colimit of a small diagram of affine schemes $S_\alpha$. By assumption, we have that the functor
\[
\Mod_{\QCoh(X)}  \rightarrow \lim \Mod_{\QCoh(S_\alpha)}
\]
is fully faithful, and in particular for every $\QCoh(X)$-module $\dcal$ we have an equivalence
\[
\dcal = \lim \dcal \otimes_{\QCoh(X)} \QCoh(S_\alpha).
\]
Specializing to free modules $\dcal = \ecal \otimes \QCoh(X)$ we deduce
\[
\ecal \otimes \QCoh(X) = \lim \ecal \otimes \QCoh(S_\alpha).
\]
The above equivalence is functorial in $\ecal$, so it defines an equivalence
\[
- \otimes \QCoh(X) = \lim - \otimes \QCoh(S_\alpha)
\]
of endofunctors of $\Pr_\st$. Our task is now reduced to showing that the copresheaf
\[
\lim \Hom_{\Pr_\st}(\ccal, - \otimes \QCoh(S_\alpha))
\]
is corepresented. Since $\Pr_{\st}$ admits small colimits it suffices to show that for every $\alpha$ the copresheaf $\Hom_{\Pr_\st}(\ccal, - \otimes \QCoh(S_\alpha))$ is corepresented. This follows from the fact that $\QCoh(S_{\alpha})$ is dualizable.
\end{proof}

Proposition \ref{proposition 1affine implies dualizable} admits a variant which applies to categories of sections of dualizable sheaves of categories.

\begin{notation}
Let $X$ be a prestack. We denote by
\[
\Gamma(X, -) : \twoQCoh(X) \rightarrow \Pr_\st
\]
the functor right adjoint to the pullback functor $\Pr_\st = \twoQCoh(\Spec(\SS)) \rightarrow \twoQCoh(X)$.
\end{notation}

\begin{remark}
Let $X$ be a prestack. Then $\Gamma(X, -)$ admits a lift to a functor
\[
\Gamma^\enh(X, -): \twoQCoh(X) \rightarrow \Mod_{\QCoh(X)}
\]
which is right adjoint to $\Loc_X$. Suppose that $X$ is written as the colimit of a small diagram of affine schemes $S_\alpha$, so that each object $\ccal$ in $\twoQCoh(X)$ corresponds to a compatible family of $\QCoh(S_\alpha)$-modules $\ccal_\alpha$. Then we have
\[
\Gamma^\enh(X, \ccal) = \lim \ccal_\alpha
\]
where here each $\ccal_\alpha$ is seen as a module over $\QCoh(X)$ by restriction of scalars along $\QCoh(X) \rightarrow \QCoh(S_\alpha)$.
\end{remark}

\begin{corollary}\label{corollary twisted 1}
Let $X$ be a $1$-affine prestack, and let $\Ccal$ be a dualizable object of $\twoQCoh(X)$. Then $\Gamma(X, \ccal)$ is dualizable.
\end{corollary}
\begin{proof}
By proposition \ref{proposition 1affine implies dualizable} we have that $\QCoh(X)$ is proper as an algebra in $\Pr_\st$, and consequently it will suffice to show that $\Gamma(X, \ccal)$ is dualizable as a $\QCoh(X)$-module. In other words, we wish to prove that $\Gamma^\enh(X, \ccal)$ is a dualizable object of $\Mod_{\QCoh(X)}$. This follows from the fact that $\ccal$ is dualizable, since $X$ is $1$-affine.
\end{proof}

\begin{corollary}\label{corollary twisted 2}
Let $X$ be a $1$-affine prestack, and let $\mathcal{G}$ be a $\GL_1$-gerbe on $X$. Then the category $\QCoh(X, \mathcal{G})$ of quasi-coherent sheaves on $X$ twisted by $\mathcal{G}$ is dualizable.
\end{corollary}


\subsection{Affine stabilizers and \texorpdfstring{$0$}{0}-affineness of the diagonal} \label{subsection 0aff}

Our next goal is to explain a condition which guarantees that the functor $\Gamma^\enh(X, -)$ right adjoint to $\Loc_X$ is fully faithful.

\begin{definition}
Let $X$ be a prestack, and denote by $\Ocal(X)$ its commutative ring spectrum of functions; in other words, the limit in spectra of $\Ocal(S)$ over all maps $S \rightarrow X$ from an affine scheme. We say that  $X$ is $0$-affine if the canonical functor $\Mod_{\Ocal(X)} \rightarrow \QCoh(X)$ is an equivalence. We say that a morphism of prestacks $X \rightarrow Y$ is $0$-affine if for every affine scheme $S$ and every map $S \rightarrow Y$ we have that $X \times_Y S$ is $0$-affine.
\end{definition}

\begin{example}
Quasi-affine morphisms of prestacks are $0$-affine.
\end{example}

\begin{remark}
Let $f: X \rightarrow Y$ be a morphism of prestacks, and assume that $Y$ is an affine scheme. Then $f$ is $0$-affine if and only if $X$ is $0$-affine. Both conditions are equivalent to the assertion that $\Ocal_X$ is a compact generator of $\QCoh(X)$.
\end{remark}

\begin{warning}
It is not the case that $0$-affineness implies $1$-affineness: for instance, the prestack $\B^4 \GG_{a, \QQ}$ is $0$-affine but not $1$-affine.
\end{warning}

\begin{proposition}\label{proposition properties 0affine}
Let $f: X \rightarrow Y$ be a $0$-affine morphism of prestacks. Then:
\begin{enumerate}[\normalfont \indent (i)]
\item The canonical functor $\Mod_{f_* \Ocal_X}(\QCoh(Y)) \rightarrow \QCoh(X)$ is an equivalence.
\item $\QCoh(X)$ is dualizable as a $\QCoh(Y)$-module.
\item For every cartesian square of prestacks
\[
\begin{tikzcd}
X' \arrow{r}{f'} \arrow{d}{g'} & Y'\arrow{d}{g} \\ X \arrow{r}{f} & Y
\end{tikzcd}
\]
the commutative square of categories
\[
\begin{tikzcd}
\QCoh(X') & \arrow{l}[swap]{f'^*} \QCoh(Y') \\ \QCoh(X) \arrow{u}[swap]{g'^*} & \QCoh(Y) \arrow{u}[swap]{g^*} \arrow{l}[swap]{f^*}
\end{tikzcd}
\]
is horizontally right adjointable.
\item In the setting of \normalfont(iii), the canonical functor \[\QCoh(X) \otimes_{\QCoh(Y)} \QCoh(Y') \rightarrow \QCoh(X')\] is an equivalence.
\end{enumerate} 
\end{proposition}
\begin{proof}
Item (i) reduces to the case when $Y$ is affine, and in this situation the desired property is equivalent to the $0$-affineness of $X$. Item (ii) is a direct consequence of (i). Item (iii) reduces once again to the case when $Y$ and $Y'$ are affine schemes, in which case the resulting commutative square of categories is obtained by tensoring the functors $g^*$ and $f^*$ over $\QCoh(Y)$, and the desired conclusion follows from the fact that $f^*$ has a colimit preserving right adjoint. Finally, applying (i) and (iii) we have
\begin{align*}
\QCoh(X) \otimes_{\QCoh(Y)} \QCoh(Y') &= \Mod_{f_* \Ocal_X}(\QCoh(Y)) \otimes_{\QCoh(Y)} \QCoh(Y') \\ &=  \Mod_{g^*f_* \Ocal_X}(\QCoh(Y')) \\ &= \Mod_{f'_*\Ocal_{X'}}(\QCoh(Y')) \\ &= \QCoh(X')
\end{align*}
which proves part (iv).
\end{proof}

\begin{proposition}
Let $f: X \rightarrow Y$ and $g: Y \rightarrow Z$ be morphisms of prestacks. If $f$ and $g$ are $0$-affine then $g \circ f$ is $0$-affine.
\end{proposition}
\begin{proof}
The assertion reduces to the case when $Z$ is affine, so that $Y$ is $0$-affine.  Then it follows from item (i) in proposition \ref{proposition properties 0affine} that $f_*$ is conservative and colimit preserving, hence $\Ocal_X = f^*\Ocal_Y$ is a compact generator of $\QCoh(X)$ and $X$ is $0$-affine.
\end{proof}

\begin{warning}
There exist a morphism $f: X \rightarrow Y$ between $0$-affine prestacks which itself is not $0$-affine. For instance, this is the case for the projection $\Spec(\QQ) \rightarrow \B^4 \GG_{a, \QQ}$.
\end{warning}

\begin{proposition}\label{proposition gammaenh fully faithful}
Let $X$ be a prestack.
\begin{enumerate}[\normalfont \indent (1)]
\item If $X$ has $0$-affine diagonal then the functor $\Gamma^\enh(X, -)$ is fully faithful. 
\item If the diagonal of $X$ is representable by qcqs algebraic spaces and $\Gamma^\enh(X, -)$ is conservative then $X$ has $0$-affine diagonal.
\end{enumerate}
\end{proposition}
\begin{proof}
We begin by establishing (1). We will prove that $\Gamma^\enh(X, -)$ is fully faithful by showing that the counit of the adjunction $\Loc_X \dashv \Gamma^\enh(X, -)$ is an equivalence. Write $X$ as a colimit of a small diagram of affine schemes $S_\alpha$ so that objects in $\twoQCoh(X)$ correspond to  compatible families of $\QCoh(S_\alpha)$-modules $\ccal_\alpha$. Unwinding the definitions, we have to check that given such a family then for every index $\beta$ the canonical map
\[
\QCoh(S_\beta) \otimes_{\QCoh(X)} \lim \ccal_\alpha \rightarrow \ccal_\beta
\]
is an isomorphism.  Applying proposition \ref{proposition properties 0affine} we see that the left hand side in the above equation equals
\[
\lim \QCoh(S_\beta) \otimes_{\QCoh(X)} \ccal_\alpha = \lim \QCoh(S_\beta \times_X S_\alpha) \otimes_{\QCoh(S_\alpha)} \ccal_\alpha.
\]
The diagram involved in the above limit defines an object in $\lim \Mod_{\QCoh(S_\beta \times_X S_\alpha)}$, which by passing to limits is equivalent to $\Mod_{\QCoh(S_\beta)}$. It follows that we have an equivalence
\[
\QCoh(S_\beta \times_X S_\beta) \otimes_{\QCoh(S_\beta)} (\QCoh(S_\beta) \otimes_{\QCoh(X)} \lim \ccal_\alpha) = \QCoh(S_\beta \times_X S_\beta) \otimes_{\QCoh(S_\beta)} \ccal_\beta
\]
and a further base change along the diagonal $S_\beta \rightarrow S_\beta \times_X S_\beta$ proves the desired identity.

It now remains to establish (2). Fix a map $f: T \rightarrow X$ from an affine scheme $T$. We will show that $T$ is $0$-affine. Consider the sheaf of categories $f_*\QCoh_T$ on $X$ whose value on each affine point $S \rightarrow X$ is given by $\QCoh(S \times_X T)$. This admits a subsheaf $\ccal$ whose value on $S \rightarrow X$ is given by the full subcategory of $\QCoh(S \times_X T)$ generated under colimits by the unit. We have that $\Gamma(X ,\ccal)$ is a full subcategory of $\Gamma(X, f_*\QCoh_T) = \QCoh(T)$ which is closed under colimits and contains the unit. It follows that $\Gamma(X ,\ccal) = \Gamma(X, f_*\QCoh_T)$, and consequently $\ccal = \QCoh_T$. Hence $\QCoh(S \times_X T)$ is generated under colimits by the unit for every $S \rightarrow T$. Since $S \times_X T$ is a qcqs algebraic space we see that $\Ocal_{S \times_X T}$ is in fact a compact generator of $\QCoh(S \times_X T)$, which implies that $S \times_X T$ is $0$-affine.
\end{proof}

In the setting of qcqs algebraic stacks, the condition of $0$-affineness of the diagonal is closely related to the condition of having affine stabilizers:

\begin{proposition}\label{proposition characterize 0affine}
Let $X$ be a qcqs algebraic stack. If the diagonal of $X$ is $0$-affine then $X$ has affine stabilizers. The converse holds provided that $\Ocal_X$ is truncated.
\end{proposition}

\begin{proof}[Proof of proposition \ref{proposition characterize 0affine}]
Suppose first that the diagonal of $X$ is $0$-affine. Let $\Spec(k) \rightarrow X$ be a geometric point, and let $G$ be the corresponding stabilizer algebraic group. Then $G$ is the classical truncation of a base change of the diagonal of $X$. In particular, it admits an affine morphism to a $0$-affine prestack, and hence $G$ itself is $0$-affine. We wish to show that $G$ is in fact affine. 

Recall that we have an exact sequence of algebraic groups over $k$
\[
1 \rightarrow H \rightarrow G \rightarrow G_{\text{aff}} \rightarrow 1
\]
where $G_{\text{aff}}$ is the affinization of $G$. Here $H$ is anti-affine, so that $\tau_{\geq 0}(\Ocal(H)) = k$. Our task is to show that $H$ is trivial.

Let $e: \Spec(k) \rightarrow H$ be the unit. Applying \cite{AntieauStefanich} proposition 4.8 we see that the functor
\[
k \otimes_{\Ocal(H)} - : \Mod_{\Ocal(H)} \rightarrow \Mod_k
\]
is conservative on the full subcategory of $\Mod_{\Ocal(H)}$ on the coconnective modules. Since $H$ is a closed subgroup in $G$ we have that $H$ is $0$-affine, and consequently the functor
\[
e^* : \QCoh(H) \rightarrow \QCoh(\Spec(k))
\]
is conservative when restricted to $\QCoh(H)_{\leq 0}$. In particular, it is conservative on $\QCoh(H)^\heartsuit$, which implies that $e$ is surjective and hence $H$ is trivial, as desired.

Suppose now that $X$ has affine stabilizers and $\Ocal_X$ is truncated. Let $p: U \rightarrow X$ be an affine cover, and let $U_1 = U \times_X U$. We wish to show that $U_1$ is $0$-affine. Equivalently, we wish to show that the projection $p':U_1 \rightarrow U$ to the first factor is $0$-affine. Since $U_1$ is a qcqs algebraic space, we have that $p'_*$ preserves colimits, so we may reduce to showing that $p'_*$ is conservative.

Let $i: U_\cl \rightarrow U$ and $i': (U_1)_\cl \rightarrow U_1$ be the inclusions. To show that $p'_*$ is conservative it suffices to prove that $i^*p'_* = (p'_\cl)_* i'^*$ is conservative.  The fact that $\Ocal_X$ (and hence $U_1$) is truncated implies that $i'^*$ is conservative, so we may reduce to showing that $(p'_\cl)_*$ is conservative. In this way we may, by replacing $X$ with $X_\cl$, reduce to the case when $X$ is classical. Combining \cite{RydhApproximation} theorem A with \cite{HallRydhGroups} theorem 2.8 we may further assume that $X$ is Noetherian.

To prove that $p'_*$ is conservative it now suffices to show that the collection of functors $x^*p'_*$ is conservative, where $x: \Spec(k) \rightarrow U$ ranges over all field valued points of $U$. Applying \cite{SAG} lemma 2.6.1.3 we see that the family of pullback functors 
\[
\QCoh(U_1) \rightarrow \QCoh(\Spec(k) \times_U U_1)
\]
 is conservative. We may thus reduce to showing that the functor of global sections on $\Spec(k)\times_U U_1$ is conservative. In other words, we must show that the fibers of $p$ over field valued points are $0$-affine. Equivalently, this amounts to showing that every field valued point $x: \Spec(k) \rightarrow X$ is $0$-affine. We claim that this is in fact quasi-affine.
 
  Applying \cite{stacks-project} lemma 06RF we may factor $x$ through a reduced locally closed substack $X'$ of $X$ such that $X'$ is a gerbe over a qcqs algebraic space. The inclusion $X' \rightarrow X$ is quasi-affine, so we may reduce to showing that the resulting map $\Spec(k) \rightarrow X'$ is quasi-affine. Replacing $X$ by $X'$ we may now assume that $X$ is a gerbe over a (classical) qcqs algebraic space $S$. We claim that in this case $x$ is affine.
  
  We may factor $x$ as the composition 
  \[
  \Spec(k) \rightarrow \Spec(k) \times_S X \rightarrow X.
  \]
  Here the second map is affine, so it suffices to show that the first map is affine. In other words, we have reduced to the case when $X$ is a gerbe over $\Spec(k)$. After a field extension we may identify $X$ with the classifying stack of an affine algebraic group. It follows that $X$ has affine diagonal, and the desired assertion follows.
\end{proof}


\subsection{Criteria for \texorpdfstring{$1$}{1}-affineness} \label{subsection Bar}

We saw above that $\Loc_X$ has a fully faithful right adjoint whenever the diagonal of $X$ is $0$-affine. Our next goal is to present an extra condition which guarantees that $\Loc_X$ is an equivalence.

\begin{notation}\label{notation qcohco}
Let $X$ be a prestack. We set
\[
\QCoh_\co(X) = \colim \QCoh(S)
\]
where the right hand side is a colimit in $\Pr_\st$, indexed over the category of affine schemes over $X$. Here for every morphism $S \rightarrow T$ of affine schemes over $X$ the corresponding transition functor is given by the pushforward functor $\QCoh(S) \rightarrow \QCoh(T)$.

We denote by
\[
\nu_X : \QCoh_\co(X) \rightarrow \QCoh(X)
\]
the functor assembled from the collection of pushforward functors  $\QCoh(S) \rightarrow \QCoh(X)$.
\end{notation}

\begin{remark}
The assignment $X \mapsto \QCoh_\co(X)$ forms part of a functor from the category of prestacks into $\Pr_\st$, obtained by left Kan extension of the functor on affine schemes which sends each affine scheme $S$ to $\QCoh(S)$, and each morphism of affine schemes $S \rightarrow T$ to the pushforward functor $\QCoh(S) \rightarrow \QCoh(T)$. 
\end{remark}

\begin{remark}\label{remark base change nu}
Let $X$ be a prestack. Then $\nu_X$ has a canonical structure of morphism of $\QCoh(X)$-modules. Suppose now given a morphism of prestacks $f: Y \rightarrow X$, and assume that $X$ has $0$-affine diagonal. Then  $\nu_X \otimes_{\QCoh(X)} \QCoh(Y) = \nu_Y$.
\end{remark}

\begin{proposition}
The assignment $X \mapsto \QCoh_\co(X)$ is a cosheaf for the fpppf topology.
\end{proposition}
\begin{proof}
Note first that if $S$ and $T$ are a pair of affine schemes, then the induced map $\QCoh_{\co}(S) \amalg \QCoh_{\co}(T) \rightarrow \QCoh_\co(S \amalg T)$ is an equivalence. Furthermore, if $S$ is empty then $\QCoh_{\co}(S) = 0$ is the initial object in $\Pr_\st$. To check the cosheaf property it suffices then to show that for every affine cover of affine schemes $S \rightarrow T$ with \v{C}ech nerve $S_\bullet$ we have $|\QCoh_\co(S_\bullet)| = \QCoh_\co(T)$. Both sides of this equation are $\QCoh(T)$-modules, so by theorem \ref{theorem descent twoqcoh} it suffices to check that the equation holds after tensoring with $\QCoh(S)$. In this case the desired assertion follows from the fact that the \v{C}ech nerve of the projection $S \times_T S \rightarrow S$ is split.
\end{proof}

\begin{proposition}\label{proposition equivalences 1affine}
Let $X$ be a prestack with $0$-affine diagonal. Then the following are equivalent:
\begin{enumerate}[\normalfont \indent (a)]
\item $X$ is $1$-affine.
\item The functor $\nu_X : \QCoh_\co(X) \rightarrow \QCoh(X)$ is an equivalence.
\item The right adjoint to the functor $\nu_X : \QCoh_\co(X) \rightarrow \QCoh(X)$ is fully faithful.
\end{enumerate}
\end{proposition}

\begin{proof}
We begin by showing that (a) implies (b). Let $S \rightarrow X$ be a map from an affine scheme $S$. Then $\nu_X \otimes_{\QCoh(X)} \QCoh(S)$ is equivalent to $\nu_S$ (remark \ref{remark base change nu}), which is an equivalence since $S$ is affine. The fact that $X$ is $1$-affine implies that the functor $- \otimes_{\QCoh(X)} \QCoh(S)$ is equivalent to the pullback functor $\twoQCoh(X) \rightarrow \twoQCoh(S)$. These form a conservative family as we vary $S$, and consequently $\nu_X$ is an equivalence, as desired.

Next we prove that (b) implies (a). Since $\Gamma^\enh(X, -)$ is fully faithful (proposition \ref{proposition gammaenh fully faithful}), we only need to argue that $\Loc_X$ is conservative. Suppose that $f: \Ccal \rightarrow \Dcal$ is a morphism of $\QCoh(X)$-modules whose image in $\twoQCoh(X)$ is invertible; we wish to show that $f$ is invertible. For every morphism $S \rightarrow X$ from an affine scheme $S$ we have that $f \otimes_{\QCoh(X)} \QCoh(S)$ is invertible. Taking the colimit over all such $S$ we deduce that $f \otimes_{\QCoh(X)} \QCoh_\co(X)$ is invertible. Since $\nu_X$ is an equivalence we deduce that $f$ itself is invertible, as desired.

Condition (b) clearly implies (c). We finish by showing that (c) implies (b). Let $K$ be the kernel of $\nu_X$, so that we have an exact sequence of $\QCoh(X)$-modules
\[
0 \rightarrow K \rightarrow \QCoh_\co(X) \rightarrow \QCoh(X) \rightarrow 0.
\]
Our goal is to show that $K = 0$.

For each morphism $S \rightarrow X$ from an affine scheme $S$ we may tensor the above sequence with $\QCoh(S)$ to obtain an exact sequence
\[
0 \rightarrow K \otimes_{\QCoh(X)} \QCoh(S) \rightarrow \QCoh_\co(X) \otimes_{\QCoh(X)} \QCoh(S) \rightarrow \QCoh(S) \rightarrow 0.
\]
Here the second map is given by $\nu_X \otimes_{\QCoh(X)} \QCoh(S) = \nu_S$, which is an equivalence since $S$ is affine. Consequently $K \otimes_{\QCoh(X)} \QCoh(S) = 0$, and taking the colimit over all $S$ we deduce that $K \otimes_{\QCoh(X)} \QCoh_{\co}(X) = 0$.

 Tensoring our original sequence with $K$ we obtain an exact sequence
\[
K \otimes_{\QCoh(X)} K \rightarrow K \otimes_{\QCoh(X)} \QCoh_\co(X) \rightarrow K \rightarrow 0.
\]
Hence $K$ is a localization of $K \otimes_{\QCoh(X)} \QCoh_{\co}(X) = 0$, and thus $K$ itself is zero, as desired. 
\end{proof}

We now give a more concrete reformulation of condition (c) from proposition \ref{proposition equivalences 1affine} in the setting of qcqs algebraic stacks.

\begin{notation}
Let $f: Y \rightarrow X$ be a morphism of prestacks which is representable by qcqs algebraic spaces. We denote by $f^? : \QCoh(X) \rightarrow \QCoh(Y)$ the right adjoint to $f_*$.
\end{notation}

\begin{remark}
Let
\[
\begin{tikzcd}
X' \arrow{r}{g'} \arrow{d}{f'} & X \arrow{d}{f} \\
Y' \arrow{r}{g} & Y
\end{tikzcd}
\]
be a cartesian diagram of prestacks, and assume that $f$ is representable by qcqs algebraic spaces. Then the canonical base change natural transformation $g'_*f'^? \rightarrow f^?g_*$ is an isomorphism; this follows by passing to right adjoints the fact that the base change natural transformation $g^* f_* \rightarrow f'_*g'^*$ is an isomorphism.
\end{remark}

\begin{notation}
Let $p: U \rightarrow X$ be a morphism of prestacks which is representable by qcqs algebraic spaces. Let $U_\bullet$ be the \v{C}ech nerve of $p$, and for each $n$ denote by $p_n: U_n \rightarrow X$ the projection. For every $\mathcal{F}$ in $\QCoh(X)$ we denote by $\Bar_{p}(\mathcal{F})_\bullet$ the simplicial object of $\QCoh(X)$ with entries $(p_n)_*(p_n)^? \mathcal{F}$ and structure maps induced from the simplicial structure on $U_\bullet$.
\end{notation}

\begin{remark}\label{remark bar split}
Let $p: U \rightarrow X$ be a morphism of prestacks which is representable by qcqs algebraic spaces and let $\mathcal{F}$ be an object of $\QCoh(X)$. Then $\Bar_{p}(\mathcal{F})_\bullet$ admits a canonical augmentation map to $\mathcal{F}$, and the resulting augmented simplicial object becomes split after applying $p^?$.
\end{remark}

\begin{remark}\label{remark Bar is counit}
Let $p: U \rightarrow X$ be a morphism of prestacks which is representable by qcqs algebraic spaces, and let $U_\bullet$ be the \v{C}ech nerve of $p$. We may then form the geometric realization $|\QCoh(U_\bullet)|$ in $\Pr_\st$, where the transition maps are given by pushforward functors; alternatively, this can be computed as the totalization $\Tot \QCoh(U_\bullet)$ where the transition maps are given by ?-pullbacks. Let $\mu: |\QCoh(U_\bullet)|\rightarrow \QCoh(X)$ be the canonical map, and denote by $\mu^\text{R}$ its right adjoint. Then for every object $\mathcal{F}$ in $\QCoh(X)$ the map
\[
|\Bar_{p}(\mathcal{F})_\bullet| \rightarrow \mathcal{F} 
\]
agrees with the counit of the adjunction $\mu \dashv \mu^\text{R}$ at $\mathcal{F}$.
\end{remark}

\begin{definition}\label{definition bounded cohomological amplitude}
Let $F: \ccal \rightarrow \dcal$ be an exact functor between stable categories equipped with t-structures. We say that $F$ has cohomological amplitude bounded by $d$ if $F(\ccal_{\geq 0}) \subseteq \dcal_{\geq -d}$. We say that $F$ has bounded cohomological amplitude if it has cohomological amplitude bounded by $d$ for some $d$.
\end{definition}

\begin{proposition}\label{proposition check affineness via Bar}
Let $X$ be a qcqs algebraic stack with $0$-affine diagonal and let $p:U \rightarrow X$ be an affine cover. The following are equivalent:
\begin{enumerate}[\normalfont \indent (a)]
\item $X$ is $1$-affine.
\item For object $\mathcal{F}$ in $\QCoh(X)$ we have that $|\Bar_{p}(\mathcal{F})_\bullet| = \mathcal{F}$.
\item For every object $\mathcal{F}$ in $\QCoh(X)^\cn$ we have that $|\Bar_{p}(\mathcal{F})_\bullet| = \mathcal{F}$.
\end{enumerate}
If furthermore $p^?$ is assumed to have bounded cohomological amplitude then the above are  equivalent to:
\begin{enumerate}[\normalfont \indent (d)]
\item For every object $\mathcal{F}$ in $\QCoh(X)^\heartsuit$ we have that $|\Bar_{p}(\mathcal{F})_\bullet| = \mathcal{F}$.
\end{enumerate}
\end{proposition}

\begin{lemma}\label{lemma bound pstar}
Let $X$ be a qcqs algebraic stack and let $p: U \rightarrow X$ be an affine cover. Then $p_*$ has bounded cohomological amplitude.
\end{lemma}
\begin{proof}
Since $p^*$ is conservative and t-exact it suffices to show that $p^*p_*$ has bounded cohomological amplitude. This is equivalent to $(p_1)_*(p_2)^*$ where $p_1$ and $p_2$ are the projections $U \times_U U\rightarrow U$.  The desired assertion now follows from the fact that $p_2^*$ is t-exact, together with \cite{SAG} corollary 3.4.2.3.
\end{proof}

\begin{proof}[Proof of proposition \ref{proposition check affineness via Bar}]
Let $U_\bullet$ be the \v{C}ech nerve of $p$. Then $U_n$ is a qcqs algebraic space for every $n$, and in particular it is $1$-affine. By proposition \ref{proposition equivalences 1affine} we see that $\QCoh(U_\bullet) = \QCoh_\co(U_\bullet)$, and therefore the geometric realization of $\QCoh(U_\bullet)$ is equivalent to $\QCoh_\co(X)$. By remark \ref{remark Bar is counit}, condition (b) is then equivalent to the fully faithfulness of the right ajoint to $\nu_X$, which agrees with condition (a) by another application of proposition \ref{proposition equivalences 1affine}.

Clearly (b) implies (c) and (c) implies (d). We now show that (c) implies (b). Let $\mathcal{F}$ be an arbitrary object of $\QCoh(X)$. Then $\mathcal{F} = \colim_{m \to +\infty} \tau_{\geq -m} \Fcal$. Applying (c) for the connective objects $(\tau_{\geq -m} \Fcal )[m]$, we may reduce to showing that the canonical map of simplicial objects
\[
\colim_{m \to +\infty} \Bar_{p}(\tau_{\geq -m} \mathcal{F})_\bullet \rightarrow \Bar_{p}(\mathcal{F})_\bullet
\]
is an isomorphism. Equivalently, we wish to show that the cofiber
\[
\colim_{m \to \infty}  \Bar_p( \tau_{< -m} \mathcal{F})_\bullet
\]
vanishes. By lemma \ref{lemma bound pstar} we may assume that $p_*$ has cohomological amplitude bounded by $d$ for some $d$. Then the $n$-th entry in the above simplicial diagram belongs to $\QCoh(X)_{< -m + nd }$, and hence it vanishes in the limit $m \to \infty$.

We will finish by showing that (d) implies (c). Let $\mathcal{F}$ be an object of $\QCoh(X)^\cn$. We will show that for every $m \geq 0$ we have $\tau_{\leq m} |\Bar_p(\Fcal)_\bullet| = \tau_{\leq m} \Fcal$. Suppose that $p_*$ and $p^?$ have cohomological amplitude bounded by $d \geq 0$. Then the canonical map $\Bar_p(\Fcal)_\bullet \rightarrow \Bar_p(\tau_{\leq m + 2d(m+2)} \Fcal)_\bullet$ becomes an isomorphism on $n$-simplices for all $n \leq m+1$ after applying $\tau_{\leq m}$. It follows that this map induces an equivalence on geometric realizations after applying $\tau_{\leq m}$, so we may thus reduce to showing that
\[
\tau_{\leq m} | \Bar_p(\tau_{\leq m + 2d(m + 2)} \Fcal)_\bullet | = \tau_{\leq m}(\tau_{\leq m + 2d(m + 2)} \Fcal) .
\]
We claim that in fact
\[
| \Bar_p(\tau_{\leq m + 2d(m + 2)} \Fcal)_\bullet | = \tau_{\leq m + 2d(m + 2)} \Fcal .
\]
Indeed, this follows from our assumption, since $\tau_{\leq m + 2d(m + 2)} \Fcal$ belongs to the closure of $\QCoh(X)^\heartsuit$ under finite colimits and extensions.
\end{proof}


\ifx\inmain\undefined
\bibliographystyle{myamsalpha2}
\bibliography{References}
\fi


\section{Controlled algebraic stacks} \label{section controlled}

Let $X$ be a qcqs algebraic stack with $0$-affine diagonal. In section \ref{section 1affineness} we showed that $X$ is $1$-affine if and only if for every $\Fcal$ in $\QCoh(X)$ the comparison map $|\Bar_p(\Fcal)_\bullet| \rightarrow \Fcal$ is an isomorphism, where $p: U \rightarrow X$ is an affine cover. We will see in this section that this is the case whenever the following two conditions are satisfied:
\begin{enumerate}[\normalfont \indent (1)]
\item The functor $p^?: \QCoh(X) \rightarrow \QCoh(U)$ right adjoint to $p_*$ has bounded cohomological amplitude.
\item Pushforwards of eventually coconnective objects of $\QCoh(U)$ generate all eventually coconnective objects of $\QCoh(X)$ under colimits.
\end{enumerate}
We say that $X$ is controlled if the above conditions are satisfied. The goal of this section is to make a general study of the notion of controlled algebraic stack.

We begin in \ref{subsection def controlled} with some general estimates that show that properties (1) and (2) do not depend on the choice of $p$. We then give the definition of controlled algebraic stack, and prove two pleasant properties: controlled algebraic stacks are $1$-affine (proposition \ref{proposition controlled is 1affine}) and have cohomologically bounded products (proposition \ref{proposition bound products}). We also show that $X$ is controlled if and only if $X_\cl$ is controlled (proposition \ref{proposition classical controlled}): this allows us to restrict our attention to the case of classical algebraic stacks.

We then prove in \ref{subsection def controlled} that the property of being controlled  descends along quasi-finite flat surjective morphisms (theorem \ref{theorem quasi-finite}). Combined with our results from section \ref{section well stabilized}, this will allow us in section \ref{section proofs} to show that well stabilized stacks are controlled by first reducing to the case of very well stabilized stacks, where stratification techniques will be shown to work.


\subsection{Definition and basic properties}\label{subsection def controlled}

We begin with some basic estimates which will imply that the property of being controlled does not depend on the choice of affine cover.

\begin{notation}
Let $X$ be an algebraic stack. We say that an object in $\QCoh(X)$ is eventually coconnective if it is a shift of a coconnective object. We denote by $\QCoh(X)^\tau$ the full subcategory of $\QCoh(X)$ generated under colimits by the eventually coconnective objects. 
\end{notation}

\begin{proposition}\label{proposition controlled does not depend on U}
Let $X$ be a qcqs algebraic stack and let  $p: U \rightarrow X$ and $q: V \rightarrow X$ be affine covers. 
\begin{enumerate}[\normalfont \indent (1)]
\item If $p^?$ has cohomological amplitude bounded by $d$ then $q^?$ has cohomological amplitude bounded by $d+1$.
\item If $p_*$ has cohomological amplitude bounded by $d$ then $q_*$ has cohomological amplitude bounded by $d+1$.
\item The full subcategory of $\QCoh(X)$ generated under colimits by $p_*{\QCoh(U)^\tau}$ agrees with the full subcategory of $\QCoh(X)$ generated under colimits by $q_*{\QCoh(V)^\tau}$. 
\end{enumerate}
\end{proposition}

\begin{lemma}\label{lemma check bound on heart}
Let $\ccal$ and $\Dcal$ be presentable stable categories equipped with t-structures, and let $F: \Ccal \rightarrow \Dcal$ be a limit preserving functor. Assume that the t-structure on $\Ccal$ is left complete and that products in $\Dcal$ are t-exact. Then for each $d \geq 0$ the following are equivalent:
\begin{enumerate}[\normalfont \indent (i)]
\item $F$ has cohomological amplitude bounded by $d$.
\item $F(c)$ is $(-d)$-connective for every object $c$ in $\Ccal^\heartsuit$.
\end{enumerate}
\end{lemma}
\begin{proof}
The fact that (i) implies (ii) is clear, so we only need to show that (ii) implies (i). Let $c$ be a connective object in $\Ccal$. We wish to show that $F(c)$ is $(-d)$-connective.

Since the t-structure on $\Ccal$ is left complete and $F$ preserves limits we have isomorphisms 
\[
F(\tau_{\geq 1} c) = F(\lim \tau_{\leq m} \tau_{\geq 1} c) = \lim F(\tau_{\leq m} \tau_{\geq 1} c).
\]
For each $m$ we have that  $(\tau_{\leq m} \tau_{\geq 1} c)[-1]$ belongs to the closure of $\Ccal^\heartsuit$ under finite colimits and extensions, so we have that $F(\tau_{\leq m} \tau_{\geq 1} c)$ is $(-d+1)$-connective. Since products in $\Dcal$ are t-exact we have that inverse limits have cohomological amplitude bounded by $1$, and hence $F(\tau_{\geq 1} c)$ is $(-d)$-connective.  Consider now the exact sequence
\[
F( \tau_{\geq 1} c ) \rightarrow F(c) \rightarrow F(H_0(c)).
\]
We have shown that the first term is $(-d)$-connective, and the third term is $(-d)$-connective by our assumptions. It follows that $F(c)$ is $(-d)$-connective, as desired.
\end{proof}

\begin{lemma}\label{lemma bound amplitude}
Let $f: S \rightarrow T$ be a flat almost finitely presented morphism of affine schemes. Then $f^?$ has cohomological amplitude bounded by $1$.
\end{lemma}
\begin{proof}
By lemma \ref{lemma check bound on heart} it is enough to show that $f^?(\Fcal)$ is $(-1)$-connective for every $\Fcal$ in $\QCoh(T)^\heartsuit$. Let $i_T: T_\cl \rightarrow T$ and $i_S: S_\cl \rightarrow S$ be the inclusions. Then $\Fcal = (i_T)_* \Gcal$ for some object $\Gcal$ in $\QCoh(T_\cl)^\heartsuit$. Now  
\[
f^? \Fcal = f^? (i_T)_* \Gcal = (i_S)_* (f_\cl)^? \Gcal
\]
and since $(i_S)_*$ is t-exact we may reduce to showing that $(f_\cl)^? \Gcal$ belongs to $\QCoh(S_\cl)_{\geq -1}$. In other words, we have reduced to showing that $\pi_0(\Ocal(S))$ has projective dimension at most $1$ as a $\pi_0(\Ocal(T))$-module. This follows from the fact that $\pi_0(\Ocal(S))$ is a flat and $\omega_1$-compactly generated $\pi_0(\Ocal(T))$-module (and hence a directed colimit of a sequence of finite free modules).
\end{proof}

\begin{proof}[Proof of proposition \ref{proposition controlled does not depend on U}]
We first prove part (1). Fix an affine cover $W \rightarrow U \times_X V$. Let $q': W \rightarrow U$ and $p' : W \rightarrow V$ be the projections, and denote by $W_\bullet$ the \v{C}ech nerve of $p'$. For each $n$ let $p'_n: W_n \rightarrow V$ be the projection, and let $r_n: W_n \rightarrow W$ be the first projection.

 We have that $q^?$ is the geometric realization of $\Bar_{p'}(q^?(-))_\bullet$, so it suffices to show that each of the functors $(p'_n)_*(p'_n)^?q^?$ has cohomological amplitude bounded by $d+1$.  Since $(p'_n)_*$ is t-exact, we may further reduce to showing that for each $n \geq 1$ the functor $(p'_n)^? q^?$ has cohomological amplitude bounded by $d + 1$. This may be written as $(q' \circ r_n)^? p^?$, and since $p^?$ is assumed to have cohomological amplitude bounded by $d$ we may reduce to showing that $(q' \circ r_n)^?$ has cohomological amplitude bounded by $1$. This now follows from an application of lemma \ref{lemma bound amplitude}.
 
We now prove (2). Since $q_*$ is the geometric realization of $q_*\Bar_{p'}(-)_\bullet$ it suffices to show that each of the functors $q_*(p'_n)_*(p'_n)^?$ has cohomological amplitude bounded by $d+1$. By lemma \ref{lemma bound amplitude} we see that $(p'_n)^?$ has cohomological amplitude bounded by $1$, so we may reduce to showing that $q_* (p'_n)_*$ has cohomological amplitude bounded by $d$. Indeed, this equals $p_* q'_* (r_n)_*$, and the desired assertion follows from the fact that $q'_*$ and $(r_n)_*$ are t-exact, together with our assumption on $p_*$.
 
It remains to prove (3). It suffices by symmetry to show that for every eventually coconnective object $\Fcal$ in $\QCoh(V)$ we have that $q_* \Fcal$ belongs to the closure under colimits of $p_*\QCoh(U)^\tau$. Indeed, $q_* \Fcal$ is the geometric realization of $q_* \Bar_{p'}(\Fcal)_\bullet$, whose entries are given by 
\[
q_*(p'_n)_*(p'_n)^? \Fcal = p_*(q'_* (r_n)_* (p'_n)^?\Fcal)
\]
and the desired claim follows from the fact that $q'_* (r_n)_* (p'_n)^?\Fcal$ is eventually coconnective for all $n$.
\end{proof}

\begin{definition}\label{definition controlled}
Let $X$ be a qcqs algebraic stack with $0$-affine diagonal. We say that $X$ is controlled if for any affine cover $p: U \rightarrow X$ the following conditions are satisfied:
\begin{enumerate}[\normalfont \indent (a)]
\item $p^?$ has bounded cohomological amplitude.
\item $\QCoh(X)^\tau$ is generated under colimits by $p_*{\QCoh(U)^\tau}$.
\end{enumerate}
\end{definition}

\begin{remark}
It follows from proposition \ref{proposition controlled does not depend on U} that the validity of conditions (a) and (b) from definition \ref{definition controlled} does not depend on the choice of affine cover.
\end{remark}

\begin{remark}\label{remark controlled iff 1affine when bound}
Let $X$ be a qcqs algebraic stack with $0$-affine diagonal, and suppose that condition (a) in definition \ref{definition controlled} is satisfied, for an affine cover $p: U \rightarrow X$. If in addition $X$ is $1$-affine then we have, for any eventually coconnective object $\Fcal$ in $\QCoh(X)$, that $\Fcal$ is the geometric realization of $\Bar_p(\Fcal)_\bullet$ (proposition \ref{proposition check affineness via Bar}), whose entries belong to $p_*\QCoh(U)^\tau$. Hence we see that condition (b) is automatically satisfied. Combined with proposition \ref{proposition controlled is 1affine} below, this shows that, in the presence of condition (a), condition (b) is equivalent to $1$-affineness.
\end{remark}

\begin{example}\label{example BGm}
For each $n \geq 0$ the stack $\B \GG^n_{m, \ZZ}$ is controlled. Properties (a) and (b) are readily verified for the cover $p: \Spec(\ZZ) \rightarrow \B \GG^n_{m, \ZZ}$, by making use of the Cartier duality equivalence $\QCoh(\B \GG^n_{m, \ZZ}) = (\Mod_{\ZZ})^{\ZZ^n}$.
\end{example}

\begin{example}\label{example linearly reductive}
Let $G$ be an affine algebraic group over a field $k$. We claim that if $G$ is linearly reductive then $\B G$ is controlled. Let $p: \Spec(k) \rightarrow \B G$ be the projection. Condition (a) in definition \ref{definition controlled} follows readily from the fact that $p_*\Ocal_{\Spec(k)}$ is a projective object of $\QCoh(\B G)$. Meanwhile, condition (b) follows from the fact that if $\Fcal$ is an eventually coconnective object of $\QCoh(\B G)$ then $\Fcal$ is a direct summand of $p_*p^* \Fcal$.
\end{example}

\begin{example}\label{example quasi-affine}
Let $X$ be a quasi-affine scheme. Then it follows from lemma \ref{lemma bound amplitude} that $X$ satisfies condition (a) in definition \ref{definition controlled}. Combined with the fact that $X$ is $1$-affine, this shows that $X$ is controlled, by remark \ref{remark controlled iff 1affine when bound}.
\end{example}

Our interest in the notion of controlled algebraic stacks stems from the following:

\begin{proposition}\label{proposition controlled is 1affine}
Let $X$ be a qcqs algebraic stack with $0$-affine diagonal. If $X$ is controlled then $X$ is $1$-affine.
\end{proposition}

\begin{lemma}\label{lemma functor preserves geometric realization}
Let $F: \mathcal{C} \rightarrow \mathcal{D}$ be an exact functor between stable categories and let $S_\bullet$ be a simplicial object in $\mathcal{C}$. Suppose that $\mathcal{C}$ and $\mathcal{D}$ are equipped with left complete t-structures, that $F$ has bounded cohomological amplitude, and that $S_\bullet$ factors through $\ccal_{\geq -d}$ for some $d$. Then $S_\bullet$ has a geometric realization, which is preserved by $F$. 
\end{lemma}
\begin{proof}
By shifting the t-structures we may, without loss of generality, assume that $F$ preserves connective objects and that $S_\bullet$ is a simplicial object in $\Ccal_{\geq 0}$. We have a commutative diagram
\[
\begin{tikzcd}
\ccal_{\geq 0} \arrow{r}{} \arrow{d}{F} & \ldots \arrow{r}{\tau_{\leq 2}} & \ccal_{[0,2]} \arrow{r}{\tau_{\leq 1}} \arrow{d}{\tau_{\leq 2}F} & \ccal_{[0,1]} \arrow{r}{\tau_{\leq 0}} \arrow{d}{\tau_{\leq 1}F} & \ccal^\heartsuit \arrow{d}{F^\heartsuit} \\  
\dcal_{\geq 0} \arrow{r}{\tau_{\leq 1}}  & \ldots \arrow{r}{\tau_{\leq 2}} & \dcal_{[0,2]} \arrow{r}{\tau_{\leq 1}}  & \dcal_{[0,1]} \arrow{r}{\tau_{\leq 0}} & \ccal^\heartsuit
\end{tikzcd}
\]
where the rows are limit diagrams of categories with finite colimits, every horizontal functor is a left adjoint, and every vertical functor preserves finite colimits. To prove the lemma it suffices to show that $\tau_{\leq m} S_\bullet$ has a geometric realization in $\ccal_{[0, m]}$ for all $m$, which is preserved by $\tau_{\leq m} F$. Indeed, $\ccal_{[0, m]}$ and $\dcal_{[0, m]}$ are  $(m+1,1)$-categories, so geometric realizations in them be  can computed in terms of finite colimits.
\end{proof}

\begin{lemma}\label{lemma uniform bound pull push}
Let $X$ be a qcqs algebraic stack with $0$-affine diagonal. If $X$ is controlled then there exists an integer $d \geq 0$ such that for every qcqs algebraic space $U$ and every flat almost finitely presented morphism $p: U \rightarrow X$ the functor $p_*p^?$ has cohomological amplitude bounded by $d$.
\end{lemma}
\begin{proof}
By proposition \ref{proposition controlled does not depend on U} we may pick $d \geq 0$ such that $p_*p^?$ has cohomological amplitude bounded by $d$ for every affine cover $p: U \rightarrow X$. We claim that we have the same bounds also in the case when $U$ is an algebraic space and $p$ is flat and almost finitely presented.

First, suppose that $U$ is affine (but $p$ is not necessarily surjective). Let $W \rightarrow X$ be an affine cover and let $p': U \amalg W \rightarrow X$ be the canonical map. Then $p_*p^?$ is a direct summand of $p'_*p'^?$, so the desired bound on $p_*p^?$ follows from the case of affine covers.

Suppose now that $U$ is a qcqs algebraic space. Let $q: V \rightarrow U$ be an affine cover. Denote by $V_\bullet$ the \v{C}ech nerve of $q$, and for each $n$ let $q_n: V_n \rightarrow U$ be the projection. Then applying proposition \ref{proposition check affineness via Bar} we see that $p_*p^?$ is the geometric realization of $p_*\Bar_q(p^?(-))_\bullet$, whose entries are given by  $(p \circ q_n)_* (p\circ q_n)^?$. Assume for a moment that $U$ has affine diagonal. Then each of the entries in $p_*\Bar_q(p^?(-))_\bullet$ has cohomological amplitude bounded by $d$ by the affine case of the proposition, so we deduce that $p_*p^?$ has cohomological amplitude bounded by $d$. In the general case the same reasoning now applies, using the fact that $V_n$ has affine diagonal for every $n$ (since it is a quasi-affine scheme).
\end{proof}

\begin{proof}[Proof of proposition \ref{proposition controlled is 1affine}]
Let $p: U \rightarrow X$ be an affine cover.  By proposition \ref{proposition check affineness via Bar} it is enough to show that for every object $\mathcal{F}$ in $\QCoh(X)^\heartsuit$, the natural map
\[
|\Bar_p(\mathcal{F})_\bullet| \rightarrow \mathcal{F}
\]
is an isomorphism. Observe that this map belongs to $\QCoh(X)^\tau$. Since $X$ is controlled we have that $p^?$ is conservative on $\QCoh(X)^\tau$, so we may reduce to showing that the induced map
\[
p^? |\Bar_p(\Fcal)_\bullet| \rightarrow p^?\mathcal{F}
\]
is an isomorphism. Applying lemma \ref{lemma uniform bound pull push} we may find an integer $d \geq 0$ such that $\Bar_p(\Fcal)_\bullet$ factors through $\QCoh(X)_{\geq -d}$.   An application of lemma \ref{lemma functor preserves geometric realization} then shows that $p^?$ preserves the geometric realization of $\Bar_p(\Fcal)_\bullet$. Consequently, it is enough to prove that the canonical map
\[
|p^? \Bar_p(\Fcal)| \rightarrow p^? \Fcal
\]
is an isomorphism. This follows from remark \ref{remark bar split}.
\end{proof}

We also have good cohomological behavior of products on controlled algebraic stacks:

\begin{proposition}\label{proposition bound products}
Let $X$ be a qcqs algebraic stack with $0$-affine diagonal. If $X$ is controlled then products in $\QCoh(X)$ have bounded cohomological amplitude.
\end{proposition}

\begin{lemma}\label{lemma bound pull push product}
Let $X$ be a qcqs algebraic stack with $0$-affine diagonal. If $X$ is controlled then there exists an integer $d \geq 0$ such that for every qcqs algebraic space $U$, every flat almost finitely presented morphism $p: U \rightarrow X$, and every family $\Fcal_\alpha$ of connective objects of $\QCoh(X)$, the object $p_*p^? \prod_\alpha \Fcal_\alpha$ is $(-d)$-connective.
\end{lemma}
\begin{proof}
By proposition \ref{proposition controlled does not depend on U} we may pick an even integer $d \geq 0$ such that $q_*$ and $q^?$ have cohomological amplitude bounded by $d/2$ for every affine cover $q: V \rightarrow X$. We claim that this integer satisfies the desired conclusion. For the remainder of the proof we fix $p: U \rightarrow X$ and $\Fcal_\alpha$ as in the statement.

Suppose first that $p$ is affine. Let $W \rightarrow X$ be an affine cover and let $p': U \amalg W \rightarrow X$ be the canonical map. Then $p_*p^? \prod_\alpha \Fcal_\alpha$ is a direct summand of
\[
p'_*p'^? \textstyle{\prod}_\alpha \Fcal_\alpha = p'_* \textstyle{\prod}_\alpha p'^? \Fcal_\alpha
\]
and the desired bound follows from our bounds on $p'_*$ and $p'^?$ together with the fact that products in $\QCoh(U \amalg W)$ are t-exact.

Suppose now that $U$ is a qcqs algebraic space. Let $q: V \rightarrow U$ be an affine cover. Denote by $V_\bullet$ the \v{C}ech nerve of $q$, and for each $n$ let $q_n: V_n \rightarrow U$ be the projection. Then $p_*p^? \prod_\alpha \Fcal_\alpha$ is the geometric realization of $p_*\Bar_q(p^? \prod_\alpha \Fcal_\alpha)_\bullet$, whose entries are given by  $(p \circ q_n)_* (p\circ q_n)^? \prod_\alpha \Fcal_\alpha$. Assume for a moment that $U$ has affine diagonal. Then each of the entries in $p_*\Bar_q(p^?\prod_\alpha \Fcal_\alpha)_\bullet$  is $(-d)$-connective by the affine case of the proposition, so we deduce that $p_*p^?\prod_\alpha \Fcal_\alpha$  is $(-d)$-connective.  In the general case the same reasoning now applies, using the fact that $V_n$ has affine diagonal for every $n$ (since it is a quasi-affine scheme).
\end{proof}

\begin{proof}[Proof of proposition \ref{proposition bound products}]
Let $d$ be as in lemma \ref{lemma bound pull push product}. We claim that products in $\QCoh(X)$ have cohomological amplitude bounded by $d$. Let $\Fcal_\alpha$ be a family of connective objects in $\QCoh(X)$ and let $p: U \rightarrow X$ be an affine cover. Then 
\[
\textstyle{\prod}_\alpha \Fcal_\alpha = |\Bar_p(\textstyle{\prod}_\alpha \Fcal_\alpha)_\bullet|
\]
so it is enough to show that each entry in $\Bar_p(\prod_\alpha \Fcal_\alpha)_\bullet$  is $(-d)$-connective.   Let $U_\bullet$ be the \v{C}ech nerve of $p$, and for each $n$ denote by $p_n: U_n \rightarrow X$ the projection. Then the entries in $\Bar_p(\prod_\alpha \Fcal_\alpha)_\bullet$ are given by $(p_n)_*(p_n)^?\prod_\alpha \Fcal_\alpha$, and the desired bound follows from our requirement on $d$.
\end{proof}

The remainder of this section is devoted to establishing some basic results concerning the class of controlled stacks.

\begin{proposition}\label{proposition classical controlled}
Let $X$ be a qcqs algebraic stack with $0$-affine diagonal. If $X_\cl$ is controlled then $X$ is controlled.
\end{proposition}
\begin{proof}
Let $p: U \rightarrow X$ be an affine cover, and denote by $i_X : X_\cl \rightarrow X$ and $i_U: U_\cl \rightarrow U$ the inclusions. We wish first to show that $p^?$ has bounded cohomological amplitude. Suppose that $(p_\cl)^?$ has cohomological amplitude bounded by $d$. We claim that $p^?$ also has cohomological amplitude bounded by $d$. By lemma \ref{lemma check bound on heart} it suffices to show that for every object $\Fcal$ in $\QCoh(X)^\heartsuit$ we have that $p^?\Fcal$ is $(-d)$-connective.  Write $\mathcal{F} = (i_X)_* \mathcal{G}$ for some object $\mathcal{G}$ in $\QCoh(X_\cl)^\heartsuit$. Then 
\[
p^? \mathcal{F} = p^? (i_X)_* \mathcal{G} = (i_U)_* (p_\cl)^? \mathcal{G}
\]
and the desired assertion follows from the fact that $(p_\cl)^?$ has cohomological amplitude bounded by $d$ and $(i_U)_*$ is t-exact.

It now remains to show that $p_*\QCoh(U)^\tau$ generates $\QCoh(X)^\tau$ under colimits. Since $(i_U)_* \QCoh(U_\cl)^\tau$ is contained in $\QCoh(U)^\tau$, it suffices to prove that $p_*(i_U)_* \QCoh(U_\cl)^\tau$ generates  $\QCoh(X)^\tau$ under colimits. We have  
\[
p_*(i_U)_* \QCoh(U_\cl)^\tau = (i_X)_* (p_\cl)_* \QCoh(U_\cl)^\tau,
\]
 so using the fact that $X_\cl$ is controlled we may further reduce to showing that $(i_X)_* \QCoh(X_\cl)^\tau$ generates $\QCoh(X)^\tau$ under colimits. Indeed, $\QCoh(X)^\tau$ is generated under colimits by shifts of objects in $\QCoh(X)^\heartsuit$, which belong to $(i_X)_*\QCoh(X_\cl)^\tau$ since $(i_X)_*$ induces an equivalence on hearts.
\end{proof}

\begin{proposition}\label{proposition quasi-affine over controlled}
Let $f: X \rightarrow Y$ be a quasi-affine morphism of qcqs algebraic stacks with $0$-affine diagonal. If $Y$ is controlled then $X$ is controlled.
\end{proposition}

\begin{lemma}\label{lemma bound on global}
Let $S$ be a quasi-affine scheme and let $\Fcal$ be a quasi-coherent sheaf on $S$. If $\Gamma(S, \Fcal)$ is connective then $S$ is connective.
\end{lemma}
\begin{proof}
We have a quasi-compact open immersion $j: S \rightarrow T$ for some affine scheme $T$.  Our hypothesis guarantees that $j_*$ is connective. It follows that $j^*j_*\Fcal = \Fcal$ is connective, as desired.
\end{proof}

\begin{proof}[Proof of proposition \ref{proposition quasi-affine over controlled}]
Let $p: U \rightarrow Y$ be an affine cover. Let $p': U' \rightarrow X$ and $f': U' \rightarrow U$ be the base changes of $p$ and $f$. Let $q: V \rightarrow U'$ be an affine cover, so that $p' \circ q$ is an affine cover of $X$.

 We first show that $(p'\circ q)^?$ has bounded cohomological amplitude. Using example \ref{example quasi-affine} we reduce to showing that $p'^?$ has bounded cohomological amplitude. By lemma \ref{lemma bound on global} it suffices to show that $f'_* p'^?$ has bounded cohomological amplitude. Indeed, this is equivalent to $p^?f_*$, which is the composition of two functors of bounded cohomological amplitude.

By remark \ref{remark controlled iff 1affine when bound}, condition (b) in definition \ref{definition controlled} will follow if we show that $X$ is $1$-affine. This follows from the fact that $Y$ is $1$-affine (by proposition \ref{proposition controlled is 1affine}) and $f$ is $1$-affine.
\end{proof}


\subsection{Quasi-finite flat locality}

Our next goal is to prove the following:

\begin{theorem}\label{theorem quasi-finite}
Let $f_\alpha: X_\alpha \rightarrow X$ be a jointly surjective family of maps between qcqs algebraic stacks with $0$-affine diagonal. Suppose that each $f_\alpha$ is quasi-finite, flat, separated, of finite presentation to order $0$, and representable by algebraic spaces. If $X_\alpha$ is controlled for all $\alpha$ then $X$ is controlled.
\end{theorem}

The proof of theorem \ref{theorem quasi-finite} will require some preliminaries.

\begin{lemma}\label{lemma union of controlled}
Let $X_\alpha$ be a finite family of qcqs algebraic stacks with $0$-affine diagonal. If each $X_\alpha$ is controlled then their disjoint union $X = \coprod X_\alpha$ is controlled.
\end{lemma}
\begin{proof}
Fix for each $\alpha$ an affine cover $p_\alpha: U_{\alpha} \rightarrow X_{\alpha}$. Let $U = \coprod U_\alpha$ so that we have an affine cover $p: U \rightarrow X$. We have t-exact equivalences
\[
\QCoh(U) = \prod_\alpha \QCoh(U_\alpha)
\]
and 
\[
\QCoh(X) = \prod_\alpha \QCoh(X_\alpha)
\]
which interchange the functor $p_*$ for $\prod_\alpha p_{\alpha, *}$ and the functor $p^?$ for $\prod_\alpha p^?_\alpha$. Conditions (a) and (b) in definition \ref{definition controlled} for $p$ now follow from the fact that they hold for each $p_\alpha$.
\end{proof}

\begin{definition}\label{definition finite flat}
Let $f: X \rightarrow Y$ be a morphism between algebraic stacks. We say that $f$ is finite flat if it is affine and $f_*\Ocal_X$ is a vector bundle on $Y$.
\end{definition}

\begin{warning}
Let $f: X \rightarrow Y$ be a morphism between algebraic stacks. The property of $f$ being finite flat is not the same as being finite and flat (where here $f$ is said to be finite if it is affine and $\pi_0(f_*\Ocal_X)$ is, locally on $Y$, a finitely generated $\pi_0\Ocal_Y$-module). The two notions however agree in the case when $Y$ is Noetherian.
\end{warning}

\begin{lemma}\label{lemma finite flat}
Let $f: X' \rightarrow X$ be a morphism between qcqs algebraic stacks with $0$-affine  diagonal. Suppose that $f$ is a finite flat surjection. If $X'$ is controlled then $X$ is controlled.
\end{lemma}
\begin{proof}
Let $p: U \rightarrow X$ be an affine cover. We wish first to show that $p^?$ has bounded cohomological amplitude. Let $p': U' \rightarrow X'$ and $f': U' \rightarrow U$ be the base changes of $p$ and $f$. The fact that $f'$ is a finite flat surjection implies that $f'^?$ is conservative and t-exact. We may thus reduce to showing that $f'^? p^?$ has bounded cohomological amplitude. This agrees with the composition $p'^? f^?$. Here $f^?$ is t-exact and $p'^?$ has bounded cohomological amplitude since $X$ is controlled, so our claim follows.

To finish the proof it remains to show that $p_*\QCoh(U)^\tau$ generates $\QCoh(X)^\tau$ under colimits. Since $p_*\QCoh(U)^\tau$ contains $p_*f'_*\QCoh(U')^\tau$, it suffices to show that the latter generates $\QCoh(X)^\tau$ under colimits. This agrees with $f_*p'_*\QCoh(U')^\tau$ and, since $X'$ is controlled, its colimit closure agrees with the colimit closure of $f_*\QCoh(X')^\tau$. We have thus reduced to showing that $f_*\QCoh(X')^\tau$ generates $\QCoh(X)^\tau$ under colimits. 

Let $\mathcal{F}$ be an eventually coconnective object in $\QCoh(X)$. Then the simplicial object $\Bar_{f}(\mathcal{F})_\bullet$ factors through $f_*\QCoh(X')^\tau$, so it is enough to prove that $|\Bar_f(\mathcal{F})_\bullet| = \mathcal{F}$. Since $f^?$ is conservative and preserves colimits we may reduce to proving that $|f^?\Bar_f(\mathcal{F})_\bullet| = f^?\mathcal{F}$. This follows from remark \ref{remark bar split}.
\end{proof}

\begin{notation}
Let $i: Z \rightarrow X$ be a closed immersion of classical prestacks, and assume that the complementary open to $Z$ is quasi-compact. Let $\hat{i}: X^\wedge_Z \rightarrow X$ be the inclusion of the formal completion of $X$ at $Z$. Then we denote by $(-)^\wedge_Z: \QCoh(X) \rightarrow \QCoh(X)$ the functor of completion at $Z$, that is, the composite functor
\[
\QCoh(X) \xrightarrow{\hat{i}^*} \QCoh(X^\wedge_Z) \xrightarrow{\hat{i}_*} \QCoh(X).
\]
\end{notation}

\begin{remark}\label{remark formal completion as limit}
Let $i: Z \rightarrow X$ be a closed immersion of classical algebraic stacks, and assume that the ideal sheaf $\Ical$ of $Z$ is finitely generated. For each $n$ let $Z_n = \Spec(\Ocal_X/\Ical^n)$, with corresponding inclusion $i_n: Z_n \rightarrow X$. Then we have $X^\wedge_Z = \colim Z_n$, and in particular we have an equivalence
\[
(-)^\wedge_Z = \lim (i_n)_* (i_n)^*
\]
of endofunctors of $\QCoh(X)$.
\end{remark}

\begin{lemma}\label{lemma open and closed}
Let $X$ be a classical qcqs algebraic stack with $0$-affine diagonal. Let $Z$ be a classical closed subscheme of $X$ with finitely generated ideal sheaf, and let $Y$ be its complement. Let $\hat{i}: X^\wedge_Z \rightarrow X$ be the inclusion of the formal completion of $Z$. Assume that $Z$ and $Y$ are controlled and that the functor $(-)^\wedge_Z : \QCoh(X) \rightarrow \QCoh(X)$ has bounded cohomological amplitude. Then $X$ is controlled.
\end{lemma}
\begin{proof}
 Let $i:Z \rightarrow X$ and $j: Y \rightarrow X$ be the inclusions. Fix an affine cover $p:U \rightarrow X$. Let $p_Z : U_Z \rightarrow Z$ and $p_Y : U_Y \rightarrow Y$ be its base changes, and let $i_U: U_Z \rightarrow U$ and $j_U : U_Y \rightarrow U$ be the inclusions.   
 
We first show that $p^?$ has bounded cohomological amplitude. We have an extension of exact functors 
\[
p^? j_*  j^? \rightarrow p^? \rightarrow p^? (-)^\wedge_Z
\]
so it suffices to show that the first and the third functors have bounded cohomological amplitude.

The functor $p^?j_*j^?$ is equivalent to $(j_U)_* (p_Y)^? j^?$. We claim that all three of these functors have bounded cohomological amplitude. The fact that $(j_U)_*$ has bounded cohomological amplitude follows since $U_Y$ is a quasi-compact open in $U$. We now show that $(p_Y)^?$ has bounded cohomological amplitude. Let $q: V \rightarrow U_Y$ be an affine cover. Denote by $V_\bullet$ its \v{C}ech nerve, and by $q_n: V_n \rightarrow U_Y$ the projections. Then $(p_Y)^?$ is the geometric realization of $\Bar_{q}((p_Y)^?(-))_\bullet$, so it is enough to argue that the functors $(q_n)_*(q_n)^?(p_Y)^?$ have uniformly bounded cohomological amplitude. Indeed, $(q_n)_*$ is t-exact, and the functors $(p_Y \circ q_n)^?$ have uniformly bounded cohomological amplitude using that $Y$ is controlled together with proposition \ref{proposition controlled does not depend on U}.

To finish proving that $p^?j_*j^?$ has bounded cohomological amplitude we will show that $j^?$ has bounded cohomological amplitude. Since $j^? = j^*j_*j^?$ and $j^*$ is t-exact, it is enough to prove that $j_*j^?$ has bounded cohomological amplitude. Indeed, this sits in an exact sequence
\[
j_*  j^? \rightarrow \id \rightarrow (-)^\wedge_Z
\]
and the third term has bounded cohomological amplitude by our hypothesis.

It now remains to show that the functor $p^? (-)^\wedge_Z$ has bounded cohomological amplitude. Suppose that $p_Z^?$ has cohomological amplitude bounded by $d$. We will prove that $p^?(-)^\wedge_Z$ has cohomological amplitude bounded by $d+1$. Let $i_n: Z_n \rightarrow X$ be as in remark \ref{remark formal completion as limit}, so that $(-)^\wedge_Z = \lim (i_n)_* (i_n)^*$. We then have that $p^? (-)^\wedge_Z$ is the limit of the functors $p^?(i_n)_*(i_n)^*$ so it suffices to show that each of these has cohomological amplitude bounded by $d$. Since $(i_n)^*$ is right t-exact, we may further reduce to proving that $p^? (i_n)_*$ has cohomological amplitude bounded by $d$. By lemma \ref{lemma check bound on heart}, it is enough to check that $p^?(i_n)_*\Fcal$ is $(-d)$-connective for every object $\Fcal$ in $\QCoh(Z_n)^\heartsuit$. Since $(i_n)_* \Fcal$ has a finite filtration with associated graded pieces in $i_*\QCoh(Z)^\heartsuit$ we may reduce to the case $n = 1$. In this case the desired assertion follows from the fact that $p^?i_* = (i_U)_* (p_Z)^?$ has cohomological amplitude bounded by $d$.

To finish the proof of the lemma we have to show that $p_*\QCoh(U)^\tau$ generates $\QCoh(X)^\tau$ under colimits. Note that $p_*\QCoh(U)^\tau$ contains $i_*(p_Z)_*\QCoh(U_Z)^\tau$ and $j_*(p_Y)_*\QCoh(U_Y)^\tau$. Since $Z$ and $Y$ are controlled, we may reduce to showing that the union of $i_*\QCoh(Z)^\tau$ and $j_*\QCoh(Y)^\tau$ generate $\QCoh(X)^\tau$ under colimits.

Let $\mathcal{F}$ be an eventually coconnective object in $\QCoh(X)$. We have an exact sequence
\[
\mathcal{F}_Z \rightarrow \mathcal{F} \rightarrow \mathcal{F}_Y = j_*j^*\mathcal{F}.
\]
The third term belongs to $j_*\QCoh(Y)^\tau$ so it is enough to prove that $\mathcal{F}_Z$ belongs to the closure of $i_*\QCoh(Z)^\tau$ under colimits. Since $\mathcal{F}_Z$ is eventually coconnective, we may reduce to showing that for each $k$ the sheaf $H_k(\mathcal{F}_Z)$ belongs to the closure of $i_*\QCoh(Z)^\tau$. Replacing $\mathcal{F}$ with $H_k(\mathcal{F}_Z)$ we may now assume that $\Fcal$ belongs to $\QCoh(X)^\heartsuit_Z$.  We then have
\[
\Fcal = \colim (i_n)_* H_0((i_n)^?  \Fcal)
\]
so replacing $\Fcal$ with $(i_n)_*H_0((i_n)^?  \Fcal)$ we may now assume that $\Fcal$ belongs to $(i_n)_*\QCoh(Z_n)^\heartsuit$. In this case $\Fcal$ belongs to the closure of $i_*\QCoh(Z)^\heartsuit$ under extensions.
\end{proof}

\begin{proof}[Proof of theorem \ref{theorem quasi-finite}]
By quasi-compactness of $X$ we may restrict to the case when the family is finite.  By lemma \ref{lemma union of controlled} we have that $X' = \amalg X_\alpha$ is controlled. Note that the induced surjection $f: X' \rightarrow X$ is quasi-finite, flat, separated, of finite presentation to order $0$, and representable by algebraic spaces, so that we have in effect reduced to the case when the family consists of a single map.

By proposition \ref{proposition classical controlled} it suffices to show that $X_\cl$ is controlled. An application of proposition \ref{proposition quasi-affine over controlled} shows that $X'_\cl$ is controlled, so that replacing $X$ by $X_\cl$ and $X'$ by $X'_\cl$ we may now assume that $X$ is classical. 

We argue via quasi-finite flat d\'{e}vissage, see \cite{RydhDevissage} theorem 6.1. Condition (D1) follows from proposition \ref{proposition quasi-affine over controlled}, while condition (D2) (which is really about finite flat maps) is given by lemma \ref{lemma finite flat}. It only remains to check condition (D3). That is, suppose that we have a quasi-compact open substack $U$ of $X$ and an \'{e}tale morphism $p: V \rightarrow X$ which is separated, of finite presentation to order $0$, representable by algebraic spaces, and whose base change to the complement of $U$ is an isomorphism. We want to show that if $U$ and $V$ are controlled then $X$ is controlled.

By the main theorem of \cite{RydhApproximationSheaves} we may find a classical closed substack $Z$ of $X$ with finitely generated ideal sheaf and whose complement is $U$. Applying proposition \ref{proposition quasi-affine over controlled} we see that $Z$ is controlled, as  $Z = p^{-1}(Z)$ is a closed substack of $V$. By lemma \ref{lemma open and closed} it is now enough to show that the functor $(-)^\wedge_Z: \QCoh(X) \rightarrow \QCoh(X)$ has bounded cohomological amplitude. We may factor this functor as the composition
\[
\QCoh(X) \xrightarrow{p^*} \QCoh(V) \xrightarrow{(-)^\wedge_{p^{-1}(Z)}} \QCoh(V) \xrightarrow{p_*} \QCoh(X)
\]
where the middle arrow is the functor of completion at $p^{-1}(Z)$. Since the functors $p^*$ and $p_*$ have bounded cohomological amplitude we may reduce to showing that ${(-)^\wedge_{p^{-1}(Z)}}$ has bounded cohomological amplitude. This follows from a combination of remark \ref{remark formal completion as limit} and proposition \ref{proposition bound products}. 
\end{proof}


\ifx\inmain\undefined
\bibliographystyle{myamsalpha2}
\bibliography{References}
\fi


\section{Proofs of the main results}\label{section proofs}

In this section we provide proofs of the results advertised in the introduction to the paper.

We begin in \ref{subsection 1aff for well stab} by proving the positive direction of (the spectral version of) theorems \ref{theorem main} and \ref{theorem main affineness}, as well as theorem \ref{theorem main twisted}. We deduce this from the fact that every well stabilized Noetherian algebraic stack with $0$-affine and separated diagonal is controlled (theorem \ref{theorem controlled}). We include here also variants of this assertion which work without Noetherianity and without separatedness of the diagonal.

We then complete the proof of theorems \ref{theorem main} and \ref{theorem main affineness} in \ref{subsection non dualizability}. This is obtained as a consequence of the fact that if $k$ is a positive characteristic field then $\QCoh(\B G_{a, k})$ is not dualizable (theorem \ref{theorem non dualizability bga}).  Finally, in \ref{subsection indcoh} we give a proof of theorem \ref{theorem main indcoh} on compact generation of the category of ind-coherent sheaves, and as an application we deduce theorems \ref{theorem main derived} and \ref{theorem main products}.


\subsection{Control for well stabilized stacks} \label{subsection 1aff for well stab}

The goal of this section is to prove the following:

\begin{theorem}\label{theorem controlled}
Let $X$ be a Noetherian algebraic stack with $0$-affine diagonal. Assume that $X$ is well stabilized and has separated diagonal, or that $X$ is very well stabilized. Then $X$ is controlled.
\end{theorem}

As a consequence of theorem \ref{theorem controlled}, we obtain the following:

\begin{corollary}\label{corollary consequences controlled}
Let $X$ be a Noetherian algebraic stack with $0$-affine diagonal. Assume that $X$ is well stabilized and has separated diagonal, or that $X$ is very well stabilized. Then:
\begin{enumerate}[\normalfont \indent (1)]
\item $\QCoh(X)$ is dualizable, and its dual is equivalent to $\QCoh(X)$.
\item $X$ is $1$-affine.
\item Let $\Ccal$ be a dualizable object of $\twoQCoh(X)$. Then $\Gamma(X, \Ccal)$ is dualizable.
\item Products in $\QCoh(X)$ have bounded cohomological amplitude.
\end{enumerate}
\end{corollary}
\begin{proof}
Item (2) and (4) follow by combining theorem \ref{theorem controlled} with propositions \ref{proposition controlled is 1affine} and \ref{proposition bound products}. Item (3) then follows from (2) together with corollary \ref{corollary twisted 1}. It remains to establish (1). The fact that $\QCoh(X)$ is dualizable follows from (3), so we only need to show that it is self dual. By proposition \ref{proposition equivalences 1affine} we have an equivalence $\QCoh(X) = \QCoh_{\co}(X)$, so it suffices to show that the dual of $\QCoh_{\co}(X)$ is $\QCoh(X)$. Indeed, we have
\[
\QCoh_\co(X)^\vee = (\colim \QCoh(S))^\vee = \lim \QCoh(S)^\vee = \lim \QCoh(S) = \QCoh(X) 
\]
where the (co)limits above are taken over all affine schemes $S$ equipped with a map to $X$.
\end{proof}

\begin{remark}\label{remark descendability}
Let $X$ be a qcqs algebraic stack with $0$-affine diagonal and let $p: U \rightarrow X$ be an affine cover.  Then the condition that $X$ be $1$-affine is closely related to the condition that $p_*\Ocal_U$ be a descendable algebra in $\QCoh(X)$, which was studied in \cite{HallRemarks} and \cite{Jiang}. The latter condition is independent of the choice of $p$, and implies $1$-affineness. The converse however does not hold: for instance, the classifying stack of a finite group over a field is $1$-affine, but does not have descendability for covers in general.

The difference between these two conditions disappears when $X$ is Noetherian and very well stabilized: in this case, the fact that $\Gamma(X, -)$ has bounded cohomological amplitude and $X$ is controlled may be used to show that $\Ocal_X$ is a retract of $\colim_{\Delta^\op_{\leq d}} \Bar_p(\Ocal_X)_\bullet|_{\Delta^\op_{\leq d}}$ for some $d$, which implies descendability of $p_*\Ocal_U$. Conversely, one may show that if $X$ is Noetherian with separated diagonal and $p_*\Ocal_U$ is descendable, then $X$ is very well stabilized. To see this, one uses corollary \ref{corollary not 1 affine} below to show first that $X$ is well stabilized, which implies by theorem \ref{theorem controlled} that $p^?$ has bounded cohomological amplitude, from which one deduces that $\Gamma(X, -)$ has bounded cohomological amplitude using that $\Ocal_X$ belongs to the closure of $p_*\Ocal_U$ under finite limits, tensor products, and retracts (because of descendability).

\end{remark}

The conditions of theorem \ref{theorem controlled} may be weakened in a variety of situations. In particular, the conclusions of corollary \ref{corollary consequences controlled} extend to these cases as well.

\begin{corollary}
Let $X$ be a qcqs algebraic stack with $0$-affine diagonal over a field of characteristic zero. Then $X$ is controlled.
\end{corollary}
\begin{proof}
By proposition \ref{proposition classical controlled} we may reduce to the case when $X$ is classical. The desired result then follows by combining proposition \ref{proposition quasi-affine over controlled} with \cite{RydhApproximation} theorem A (which guarantees that $X$ admits an affine morphism to a classical qcqs algebraic stack $Y$ of finite type over $\QQ$) and \cite{HallRydhGroups} theorem 2.8 (which guarantees that $Y$ may be taken to have affine stabilizers).
\end{proof}

\begin{corollary}\label{corollary can be approximated}
Let $X$ be a qcqs algebraic stack with $0$-affine diagonal. Suppose that one of the following conditions is satisfied:
\begin{enumerate}[\normalfont \indent (i)]
\item For every geometric field valued point $x: \Spec(k) \rightarrow X$, the identity component of the stabilizer of $X$ at $x$ is of multiplicative type and its amount of connected components is invertible in $k$.
\item For every geometric field valued point $x: \Spec(k) \rightarrow X$ the reduction of the identity component of the stabilizer of $X$ at $x$ is a torus, and $X$ has separated diagonal.
\end{enumerate}
Then $X$ is controlled.
\end{corollary}

\begin{lemma}\label{lemma constructibility}
Let $S$ be a classical Noetherian affine scheme, and let $G$ be a classical group algebraic space over $S$ with affine fibers. Let $F$ be the set of points $x$ in $S$ with the property that the reduced identity component of the geometric fiber of $G$ at $x$ is a torus. Then $F$ is constructible.
\end{lemma}
\begin{proof}
By Noetherian induction it suffices to show that if $S$ is nonempty then $S$ contains a nonempty open subset $U$ which is contained either in $F$ or in the complement of $F$. It is enough to show that this is the case for the reduction of an irreducible component of $S$; in this way, we may reduce to the case when $S$ is integral. The desired assertion now follows from the arguments in the proof of proposition \ref{proposition structure G}.
\end{proof}

\begin{proof}[Proof of corollary \ref{corollary can be approximated}]
By proposition \ref{proposition classical controlled} we may reduce to the case when $X$ is classical. Using  \cite{RydhApproximation} theorem A we may write $X$ as an inverse limit of qcqs Noetherian algebraic stacks $X_\lambda$ along affine morphisms. Using \cite{HallRydhGroups} theorem 2.8 we may assume that $X_\lambda$ has $0$-affine diagonal for all $\lambda$.  By  proposition \ref{proposition quasi-affine over controlled} the result will follow if we show that $X_\lambda$ satisfies conditions (i) or (ii) for $\lambda$ large enough.

The case when $X$ satisfies condition (i) is given by \cite{HallRydhGroups} theorem 2.8. We now address the case when $X$ satisfies condition (ii). The fact that $X_\lambda$ has separated diagonal for $\lambda$ large enough is given by \cite{RydhApprox1} theorem C. Pick an index $\lambda_0$ and an affine cover $p_{\lambda_0}: U_{\lambda_0} \rightarrow X_{\lambda_0}$. We may assume without loss of generality that $\lambda_0$ is initial, so that we have base changes $p_\lambda: U_\lambda \rightarrow X_\lambda$ and $p: U \rightarrow X$. For each $\lambda$ let $G_\lambda$ be the classical group algebraic space over $U_\lambda$ underlying $U_\lambda \times_{X_\lambda \times X_\lambda} X_\lambda$, and let $F_\lambda$ be the subset $U_\lambda$ consisting of those points $x$ for which the reduced identity component of the geometric fiber of $G_\lambda$ at $x$ is a torus. We wish to show that $F_\lambda = U_\lambda$ for $\lambda$ large enough. This follows from lemma \ref{lemma constructibility}, in light of \cite{EGAIV} corollary 8.3.4.
\end{proof}

\begin{corollary}
Let $X$ be a qcqs algebraic stack with $0$-affine diagonal. Then $X$ is controlled in any of the following cases:
\begin{enumerate}[\normalfont \indent (i)]
\item $X$ is a qcqs algebraic space.
\item $X$ is a qcqs Deligne-Mumford stack with separated diagonal.
\item $X$ is a very well stabilized qcqs Deligne-Mumford stack with $0$-affine diagonal.
\item $X$ is a well stabilized qcqs algebraic stack with $0$-affine and separated diagonal over a field of positive characteristic.
\item $X$ is a very well stabilized qcqs algebraic stack with $0$-affine diagonal over a field of positive characteristic.
\end{enumerate}
\end{corollary}
\begin{proof}
Specialize corollary \ref{corollary can be approximated}.
\end{proof}

The remainder of this section is devoted to the proof of theorem \ref{theorem controlled}.

\begin{lemma} \label{lemma bounds red}
Let $X$ be a classical qcqs algebraic stack with $0$-affine diagonal and let $Z$ be a classical closed substack of $X$ whose ideal sheaf is finitely generated and nilpotent.
\begin{enumerate}[\normalfont (1)]
\item If the functor $\Gamma(Z, -): \QCoh(Z) \rightarrow \Mod_\SS$ 
has cohomological amplitude bounded by $d$ then the functor $\Gamma(X, -): \QCoh(X) \rightarrow \Mod_\SS$ has cohomological amplitude bounded by $d$.
\item Let $p: U \rightarrow X$ be an affine cover, and let $p_Z: U_Z \rightarrow Z$ be its base change. If $(p_Z)^?$ has cohomological amplitude bounded by $d$ then  $p^?$ has cohomological amplitude bounded by $d$.
\item Let $U$ be a qcqs algebraic space and let $p: U \rightarrow X$ be a flat almost finitely presented morphism. Denote by $p_Z: U_Z \rightarrow Z$ its base change to $Z$. If $(p_Z)_*(p_Z)^?$ has cohomological amplitude bounded by $d$ then $p_*p^?$ has cohomological amplitude bounded by $d$. 
\item If $Z$ is controlled then $X$ is controlled.
\end{enumerate} 
\end{lemma}
\begin{proof}
Items (1) through (3) follow readily from the fact that every connective object of $\QCoh(X)$ admits a finite filtration with quotients in the image of the pushforward functor $\QCoh(Z)^\cn \rightarrow \QCoh(X)^\cn$. Finally, item (4) is a specialization of lemma \ref{lemma open and closed} to the case when $Y$ is empty.
\end{proof}

\begin{lemma}\label{lemma open and closed part 2}
Let $X$ be a classical qcqs algebraic stack with $0$-affine diagonal. Let $Z$ be a classical closed subscheme of $X$ with finitely generated ideal sheaf, and let $Y$ be its complement. Assume that $Z$ and $Y$ are controlled and that the functor
\[
\Gamma(Z, -): \QCoh(Z) \rightarrow \Mod_\SS
\]
has bounded cohomological amplitude. Then $X$ is controlled.
\end{lemma}
\begin{proof}
By lemma \ref{lemma open and closed} it is enough to show that the functor
\[
(-)^\wedge_Z : \QCoh(X) \rightarrow \QCoh(X)
\]
has bounded cohomological amplitude. Let $p: U \rightarrow X$ be an affine cover, and denote by $p_Z : U_Z \rightarrow Z$ its base change. Let $d \geq 0$ be such that the functors $p_*$, $(p_Z)_*$ and $(p_Z)^?$ have cohomological amplitude bounded by $d$. By lemma \ref{lemma uniform bound pull push} we may, after enlarging $d$ if necessary, assume that all the entries in $\Bar_{p_Z}(-)_\bullet$ also have cohomological amplitude bounded by $d$. We claim that in this situation $(-)^\wedge_Z$ has cohomological amplitude bounded by $3d+1$.

Let $\Fcal$ be a connective object of $\QCoh(X)$. Let $i_n: Z_n \rightarrow X$ be as in remark \ref{remark formal completion as limit}, so that we have an equivalence $\Fcal^\wedge_Z = \lim (i_n)_* (i_n)^* \Fcal$. Then $\Fcal^\wedge_Z$ sits in an exact sequence
\[
\Fcal^\wedge_Z \rightarrow \textstyle{\prod} (i_n)_* (i_n)^* \Fcal \rightarrow \textstyle{\prod} (i_n)_* (i_n)^* \Fcal
\]
so it will suffice to show that $\prod (i_n)_* (i_n)^* \Fcal$ is $(-3d)$-connective.

For each $n$ denote by $p_{Z_n}: U_{Z_n} \rightarrow Z_n$ the base change of $p$. Let $\Gcal$ be the object of $\prod \QCoh(Z_n)$ with coordinates $\Ocal_{Z_n}$, and let $\Scal_\bullet$ be the simplicial object of $\prod \QCoh(Z_n)$ with coordinates $\Bar_{p_{Z_n}}(\Ocal_{Z_n})_\bullet$. By part (4) of lemma \ref{lemma bounds red} we have that $Z_n$ is controlled for all $n$, and in particular $\Gcal$ is the geometric realization of $\Scal_\bullet$. 

Part (3) of lemma \ref{lemma bounds red} guarantees that $\Scal_\bullet$ factors through $\prod \QCoh(Z_n)_{\geq -d}$. Furthermore, part (1) implies that the functor
\[
\Hom_{\prod \QCoh(Z_n)}(\Gcal, -): \textstyle{\prod} \QCoh(Z_n) \rightarrow \Mod_\SS
\]
has bounded cohomological amplitude.  An application of lemma \ref{lemma functor preserves geometric realization} now shows that $\Hom_{\prod \QCoh(Z_n)}(\Gcal, -)$ preserves the geometric realization of $\Scal_\bullet$. It follows that we have
\[
\Hom_{\prod \QCoh(Z_n)}(\Gcal, \Gcal) = \colim_{k \to \infty} \left(\Hom_{\prod \QCoh(Z_n)}\left(\Gcal, \colim_{\Delta^\op_{\leq k}} \left( \Scal_\bullet|_{\Delta^\op_{\leq k}}\right)\right)\right).
\]
Consequently, there exists $k$ such that the unit endomorphism of $\Gcal$ lifts to a map
\[
\Gcal \rightarrow \colim_{\Delta^\op_{\leq k}}\left( \Scal_\bullet|_{\Delta^\op_{\leq k}}\right),
\]
so that in particular we have that $\Gcal$ is a retract of $ \colim_{\Delta^\op_{\leq k}} \left( \Scal_\bullet|_{\Delta^\op_{\leq k}}\right)$.

Let $\Fcal'$ be the object of $\prod \QCoh(Z_n)$ with coordinates $(i_n)^* \Fcal$. Then tensoring with $\Fcal'$ we deduce that $\Fcal'$ is a retract of 
\[
\colim_{\Delta^\op_{\leq k}}\left( \Scal_\bullet|_{\Delta^\op_{\leq k}} \otimes \Fcal'\right).
\]
Composing with the functor
\[
\textstyle{\prod} \QCoh(Z_n) \xrightarrow{\prod (i_n)_*} \textstyle{\prod} \QCoh(X) \xrightarrow{\prod} \QCoh(X)
\]
we deduce that $\prod (i_n)_* (i_n)^* \Fcal$ is a retract of 
\[
\colim_{\Delta^\op_{\leq k}}\left(\textstyle{\prod} (i_n)_*\left( \Bar_{p_{Z_n}}(\Ocal_{Z_n})_\bullet|_{\Delta^\op_{\leq k}} \otimes (i_n)^* \Fcal\right) \right) .
\]
We may now reduce to showing that the entries of the simplicial object
\[
\textstyle{\prod} (i_n)_*\left( \Bar_{p_{Z_n}}(\Ocal_{Z_n})_\bullet \otimes (i_n)^* \Fcal\right) 
\]
are $(-3d)$-connective. Fix $k \geq 0$, and note that the $k$-th entry of the above simplicial object is given by
\[
\textstyle{\prod} (i_n)_* ((p_{Z_n})_* (p_{Z_n})^? (\Bar_{p_{Z_n}}(\Ocal_{Z_n})_{k-1} )\otimes (i_n)^* \Fcal )
\]
and may be rewritten, thanks to the projection formula, as
\[
\textstyle{\prod} (i_n)_* (p_{Z_n})_* ( (p_{Z_n})^? (\Bar_{p_{Z_n}}(\Ocal_{Z_n})_{k-1}) \otimes (p_{Z_n})^* (i_n)^* \Fcal).
\]
For each $n$ let $\Hcal_n$ be the pushforward along the inclusion $U_{Z_n} \rightarrow U$ of 
\[
 (p_{Z_n})^? (\Bar_{p_{Z_n}}(\Ocal_{Z_n})_{k-1}) \otimes (p_{Z_n})^* (i_n)^* \Fcal
\]
so that our task becomes showing that $\prod p_* \Hcal_n$ is $(-3d)$-connective. This  is equivalent to $p_* \prod \Hcal_n$, and since products in $\QCoh(U)$ are exact and $p_*$ has cohomological amplitude bounded by $d$, we may reduce to showing that $\Hcal_n$ is $(-2d)$-connective for every $n$. Since pushforward along $U_{Z_n} \rightarrow U$ is t-exact, and $(p_{Z_n})^* (i_n)^* \Fcal$ is connective, it suffices to show that $(p_{Z_n})^? \Bar_{p_{Z_n}}(\Ocal_{Z_n})_{k-1}$
is $(-2d)$-connective. We claim that in fact each of $\Bar_{p_{Z_n}}(-)_{k-1}$ and $(p_{Z_n})^?$ have cohomological amplitude bounded by $d$. Indeed, by lemma \ref{lemma bounds red} we may reduce to the case $n = 1$, which was part of our assumptions on $d$.
\end{proof}

\begin{lemma}\label{lemma locally on BG}
Let $S$ be an integral Noetherian affine scheme and let $G$ be an algebraic group over $S$ with affine fibers. Assume that $G$ is good. Then there exists a dense open subscheme $U$ of $S$ with the property that $\B G_U$ is controlled.
\end{lemma}
\begin{proof}
We apply proposition \ref{proposition structure G}, so that one of conditions (a) or (b) applies. If condition (a) applies then we have a quasi-affine map $\B G_U \rightarrow \B \GL_{n, \QQ}$, and hence $\B G_U$ is controlled by a combination of example \ref{example linearly reductive} and proposition \ref{proposition quasi-affine over controlled}.  Assume now that condition (b) applies. Then the projection $\B G_{U'} \rightarrow \B G_U$ is a finite flat surjection, so by lemma \ref{lemma finite flat} it suffices to show that $\B G_{U'}$ is controlled. Fix a subgroup inclusion $\mathbb{G}_{m,U'}^n \rightarrow G_{U'}$ whose quotient is finite flat. Then the induced map $\B \mathbb{G}_{m, U'}^n \rightarrow \B G_{U'}$ is a finite flat surjection. Another application of lemma \ref{lemma finite flat} reduces the problem to showing that $\B\mathbb{G}_{m, U'}^n$ is controlled. Applying proposition \ref{proposition quasi-affine over controlled} to the projection $\B\mathbb{G}_{m, U'}^n \rightarrow \B\mathbb{G}_{m,\mathbb{Z}}^n$  we may further reduce to showing that $\B\mathbb{G}_{m, \mathbb{Z}}^n$ is controlled, which was done in example \ref{example BGm}.
\end{proof}

\begin{lemma}\label{lemma controlled on open}
Let $X$ be a classical Noetherian algebraic stack with $0$-affine diagonal. If $X$ is well stabilized then there exists a dense open substack of $X$ which is controlled.
\end{lemma}
\begin{proof}
By lemma part (4) of lemma \ref{lemma bounds red} it suffices to consider the case when $X$ is reduced. By \cite{stacks-project} proposition 06RC we may, after replacing $X$ by a dense open substack, assume that $X$ is a gerbe over an algebraic space. Applying \cite{stacks-project} proposition 06NH we may further reduce to the case when $X$ is a gerbe over a (classical) affine scheme $S$. Changing base to a dense open subset of $S$ we may assume that the connected components of $S$ are irreducible. Using lemma \ref{lemma union of controlled} we may then reduce to the case when $S$ is irreducible. Applying lemma \ref{lemma bounds red} we may further assume that $S$ is integral.

Let $p: S' \rightarrow S$ be an affine cover with the property that $X$ has a section over $S'$. Changing base to a dense open subscheme of $S$ we may assume that there exists a finite flat surjection $S'' \rightarrow S$ refining $S' \rightarrow S$, whose base change to the generic point of $S$ is the spectrum of a field. Replacing $S'$ by $S''$ we may now assume that $p$ is finite flat and that the generic fiber of $p$ is the spectrum of a field. Changing base to a dense open subscheme of $S$ we may now assume that $S'$ is integral.

Let $X' = S' \times_S X'$. Then $X'$ is of the form $\B G$ for some algebraic group $G$ over $S'$ with affine fibers. By proposition \ref{proposition algebraic space over well stabilized} we have that $X'$ is well stabilized. It now follows from proposition \ref{proposition BG well stabilized iff G good} that $G$ is good. Consequently, applying lemma \ref{lemma locally on BG} we may find a dense open subscheme $U'$ of $S'$ such that $U' \times_S X$ is controlled.

Since $p$ is open we have that $p(U')$ contains the generic point of $S$, and hence the generic fiber of $p|_{U'}$ is equal to the generic fiber of $p$. Consequently, we may find an affine dense open subscheme $U$ of $S$ with the property that $p^{-1}(U)$ is contained in $U'$. Applying proposition \ref{proposition quasi-affine over controlled} we see that $p^{-1}(U) \times_S X$ is controlled. We now have that $U \times_S X$ is controlled by lemma \ref{lemma finite flat}.
\end{proof}

\begin{proof}[Proof of theorem \ref{theorem controlled}]
By proposition \ref{proposition classical controlled} we may assume that $X$ is classical. Combining theorems \ref{theorem quasi-finite structure} and \ref{theorem quasi-finite} we may further reduce to establishing the theorem in the case when $X$ is very well stabilized.

We argue by Noetherian induction. The case when $X$ is empty is clear, so assume that $X$ is nonempty and every proper classical closed substack of $X$ is known to be controlled. By lemma \ref{lemma controlled on open} we may find a nonempty open substack $Y$ of $X$ which is controlled. Let $Z$ be the reduced complement of $Y$, which we know also to be controlled by our inductive hypothesis. By \cite{HallRydhGroups} theorem 2.1 we have  that $\Gamma(Z, -)$ has bounded cohomological amplitude. The fact that $X$ is controlled now follows from lemma \ref{lemma open and closed part 2}.
\end{proof}

  
\subsection{Non dualizability for poorly stabilized stacks}\label{subsection non dualizability}

Our next goal is to prove the following:

\begin{theorem}\label{theorem non dualizability bga}
Let $k$ be a field of positive characteristic. Then the presentable stable categories $\QCoh(\B\GG_{a,k})$ and $\D(\QCoh(\B \GG_{a, k})^\heartsuit)$ are not dualizable.
\end{theorem}

As a consequence of theorem \ref{theorem non dualizability bga} we obtain converses to corollary \ref{corollary consequences controlled}.

\begin{corollary}\label{corollary non dualizability poorly stabilized}
Let $X$ be a quasi-separated algebraic stack with affine stabilizers. If $X$ is poorly stabilized then $\QCoh(X)$ is not dualizable.
\end{corollary}
\begin{proof}
By \cite{HallNeemanRydh} lemma 4.2 there exists a field $k$ of positive characteristic and a quasi-affine morphism $\B\GG_{a,k} \rightarrow X_{\cl}$. Composing with the (affine) morphism $X_\cl \rightarrow X$ we obtain a quasi-affine morphism $\B \GG_{a, k} \rightarrow X$. In this situation the corresponding pullback functor $\QCoh(X) \rightarrow \QCoh(\B \GG_{a,k})$ admits a right adjoint which is conservative and colimit preserving. The fact that $\QCoh(X)$ is not dualizable now follows from a combination of theorem \ref{theorem non dualizability bga} and \cite{Efimov} proposition 1.57.
\end{proof}

\begin{corollary}\label{corollary not 1 affine}
Let $X$ be a quasi-separated algebraic stack with affine stabilizers. If $X$ is poorly stabilized then $X$ is not $1$-affine.
\end{corollary}
\begin{proof}
Combine corollary \ref{corollary non dualizability poorly stabilized} with proposition \ref{proposition 1affine implies dualizable}.
\end{proof}

\begin{corollary}\label{coro products unbounded}
Let $X$ be a qcqs algebraic stack with affine stabilizers. If products in $\QCoh(X)$ have bounded cohomological amplitude then $X$ is well stabilized.
\end{corollary}
\begin{proof}
Using the fact that the pushforward $\QCoh(X_\cl) \rightarrow \QCoh(X)$ is t-exact, conservative and preserves products, we deduce that products in $\QCoh(X_\cl)$ have bounded cohomological amplitude. Replacing $X$ by $X_\cl$ we may now reduce to the case when $X$ is classical. By corollary \ref{corollary non dualizability poorly stabilized}, it will suffice to show that $\QCoh(X)$ is dualizable. We will do so by showing that $\QCoh(X)$ satisfies the axiom (AB6), see \cite{Efimov} section 1.9. In other words, we will show that the projection $p: \Ind(\QCoh(X)) \rightarrow \QCoh(X)$ preserves small products.

Let $\Fcal_\alpha$ be a small family of objects of $\Ind(\QCoh(X))$, and suppose that products in $\QCoh(X)$ have cohomological amplitude bounded by $d$. Our goal is to show that the comparison map $p(\textstyle{\prod}_\alpha \Fcal_\alpha) \rightarrow \textstyle{\prod}_\alpha p(\Fcal_\alpha)$ is an isomorphism. We will do so by showing that for every integer $m$ the map
\[
\tau_{\leq m}p(\textstyle{\prod}_\alpha \Fcal_\alpha) \rightarrow \tau_{\leq m}\textstyle{\prod}_\alpha p(\Fcal_\alpha)
\]
is an isomorphism. Using our bound on products, together with the t-exactness of $p$, we see that the right hand side equals $\tau_{\leq m} \prod_\alpha p(\tau_{\leq m+d}\Fcal_\alpha)$. We claim that, similarly, the left hand side equals $\tau_{\leq m}p(\textstyle{\prod}_\alpha \tau_{\leq m + d}\Fcal_\alpha) $. To see this, it is enough to show that products in $\Ind(\QCoh(X))$ have cohomological amplitude bounded by $d$. Suppose given a small family $\Gcal_\alpha$ of connective objects of $\Ind(\QCoh(X))$, and write each one as a filtered colimit of a diagram of connective objects $\Gcal_{\alpha, \beta}$ in $\QCoh(X)$. Then since (AB6) holds in $\Ind(\QCoh(X))$, we have
\[
\textstyle{\prod} \Gcal_{\alpha} = \colim \textstyle{\prod}\Gcal_{\alpha, \beta_\alpha}
\]
and the claim then follows from the fact that the objects $ \textstyle{\prod}\Gcal_{\alpha, \beta_\alpha}$ are $(-d)$-connective.

We have now reduced to showing that for every integer $m$ the map 
\[
\tau_{\leq m}p(\textstyle{\prod}_\alpha \tau_{\leq m+d}\Fcal_\alpha) \rightarrow \tau_{\leq m}\textstyle{\prod}_\alpha p(\tau_{\leq m+d}\Fcal_\alpha)
\]
is an isomorphism. We claim that, in fact, the map
\[
p(\textstyle{\prod}_\alpha \tau_{\leq m+d}\Fcal_\alpha) \rightarrow \textstyle{\prod}_\alpha p(\tau_{\leq m+d}\Fcal_\alpha)
\]
is an isomorphism. This amounts to the assertion that (AB6) is satisfied in $\QCoh(X)_{\leq m+d}$. To see this it is enough to show that $\QCoh(X)_{\leq m+d}$ is compactly generated. This is a directed colimit of categories of the form $\QCoh(X)^\cn_{\leq N}$ along compact functors, so it is enough to show that $\QCoh(X)^\cn_{\leq N}$ is compactly generated for all $N$. 

By \cite{RydhApproximation} there exists an affine morphism $X \rightarrow X'$ where $X'$ is a classical Noetherian algebraic stack. The functor of pushforward $\QCoh(X)^\cn_{\leq N} \rightarrow \QCoh(X')^\cn_{\leq N}$ is colimit preserving, conservative, and has a left adjoint, so it suffices to show that $\QCoh(X')^\cn_{\leq N}$ is compactly generated. We claim that in fact $\QCoh(X')^\cn$ is weakly coherent, in the sense of \cite{SAG} section C.6.5; this will suffice by \cite{SAG} proposition C.6.5.4. 

We apply \cite{SAG} proposition 6.5.6. Condition (i) follows from the fact that the almost compact objects of $\QCoh(X')^\cn$ are the objects which have coherent homology, and these are  closed under finite limits. Condition (ii) is a consequence of the fact that every object in $\QCoh(X')^\heartsuit$ is a filtered colimit of its coherent subobjects, see \cite{LaumonMoretBailly} proposition 15.4.
\end{proof}

The rest of this section is devoted to the proof of theorem \ref{theorem non dualizability bga}. 

\begin{notation}
For each Grothendieck abelian category $\Ccal$ we denote by $\widehat{\D}(\Ccal)$ the completed derived category of $\Ccal$. In other words, $\widehat{\D}(\Ccal)$ is obtained by left completing the canonical t-structure on the derived category $\D(\Ccal)$.
\end{notation}

\begin{lemma}\label{lemma t structure colimit}
Let $\Ccal_i$ be a filtered system of Grothendieck abelian categories, and assume that the transition functors are left exact and admit exact right adjoints. Then we have equivalences
\[
\colim \widehat{\normalfont \text{D}}(\ccal_i) = \widehat{\normalfont \text{D}}(\colim \ccal_i)
\]
and
\[
\colim \D(\ccal_i) = \D(\colim \ccal_i).
\]
\end{lemma}
\begin{proof}
We first address the assertion for the completed derived categories. It suffices to show that we have an equivalence
\[
\colim \widehat{\text{D}}(\ccal_i)_{\geq 0} = \widehat{\text{D}}(\colim \ccal_i)_{\geq 0}.
\]
Applying \cite{SAG} theorem 3.3.1 we deduce that the left hand side is a Grothendieck prestable category. Its heart is given by
\[
\Set \otimes \colim \widehat{\text{D}}(\ccal_i)_{\geq 0} = \colim( \Set \otimes \widehat{\text{D}}(\ccal_i)_{\geq 0}  ) = \colim \ccal_i.
\]
The lemma will thus follow if we are able to show that $\colim \widehat{\text{D}}(\ccal_i)_{\geq 0}$ is complete and weakly $0$-complicial (see \cite{SAG} C.5.9). 

We first show that $\colim \widehat{\text{D}}(\ccal_i)_{\geq 0}$ is complete. Since the transition functors in the diagram $\ccal_i$ admit exact right adjoints, we have that the transition functors in the diagram $\widehat{\text{D}}(\ccal_i)_{\geq 0}$ admit exact right adjoints. Passing to right adjoints we deduce that $\colim \widehat{\text{D}}(\ccal_i)_{\geq 0}$ can be written as the limit of the categories $\widehat{\text{D}}(\ccal_i)_{\geq 0}$ along exact functors. We now have
\begin{align*}
\colim \widehat{\text{D}}(\ccal_i)_{\geq 0} &= \lim \widehat{\text{D}}(\ccal_i)_{\geq 0} \\ &= \lim_i \lim_n (\widehat{\text{D}}(\ccal_i)_{\geq 0})_{\leq n} \\ &= \lim_n \lim_i  (\widehat{\text{D}}(\ccal_i)_{\geq 0})_{\leq n} \\ &= \lim_n \underset{i}{\colim} (\widehat{\text{D}}(\ccal_i)_{\geq 0})_{\leq n} 
\end{align*}
and hence $\colim (\widehat{\text{D}}(\ccal_i)_{\geq 0})$ is complete, as desired.

It remains to show that $\colim \widehat{\text{D}}(\ccal_i)_{\geq 0}$ is weakly $0$-complicial. Let $X$ be a truncated object. We have to show that $X$ receives an effective epimorphism\footnote{In other words, a morphism which induces a surjection on $H_0$.} from a $0$-truncated object. For each $i$ let $p_i:  \widehat{\text{D}}(\ccal_i)_{\geq 0} \rightarrow  \colim \widehat{\text{D}}(\ccal_i)_{\geq 0}$ be the canonical functor, and let $q_i$ be its right adjoint. Then $X = \colim p_i q_i (X)$, and in particular $X$ receives an effective epimorphism from $\bigoplus p_i q_i(X)$. We may thus reduce to showing that for every $i$, the object $p_i q_i (X)$ receives an effective epimorphism from a $0$-truncated object. Since $p_i$ is exact, we may further reduce to showing that $q_i(X)$ receives an effective epimorphism from a $0$-truncated object. Since $q_i$ is exact we have that $q_i(X)$ is truncated. The desired claim now follows from the fact that $\widehat{\text{D}}(\ccal_i)_{\geq 0}$ is weakly $0$-complicial.

The proof of the assertion at the level of ordinary (not necessarily complete) derived categories follows along the same lines as the argument above. In this case, one shows that $\colim \D(\Ccal_i)_{\geq 0}$ is a  separated and $0$-complicial Grothendieck prestable category, which implies that it is the connective derived category of its heart.
\end{proof}

\begin{lemma}\label{lemma exists endofunctor}
Let $k$ be a field of positive characteristic. Then there exists an endofunctor $F: \QCoh(\B \GG_{a,k}) \rightarrow \QCoh(\B\GG_{a,k})$, with a colimit preserving right adjoint, and with the property that the colimit  of the directed system
\[
\QCoh(\B\GG_{a,k}) \xrightarrow{F} \QCoh(\B\GG_{a,k}) \xrightarrow{F} \QCoh(\B\GG_{a,k}) \xrightarrow{F} \ldots 
\]
is equivalent to $\Mod_k$.
\end{lemma}
\begin{proof}
Consider the commutative $k$-algebra
\[
R = \frac{k[x_1, x_2, x_3, \ldots]}{(x_1^p, x_2^p, x_3^p, \ldots)}.
\]
Say that a discrete $R$-module $M$ is locally finite if every element of $M$ is annihilated by all but finitely many of the generators $x_i$, and denote by $\Mod_R^{\heartsuit, \text{lf}}$ the full subcategory of $\Mod_R^\heartsuit$ on the locally finite modules. Then $\QCoh(\B\GG_{a,k})^\heartsuit$ is equivalent to $\Mod_R^{\heartsuit, \text{lf}}$, and consequently we have that $\QCoh(\B\GG_{a,k})$ is equivalent to the completed derived category $\widehat{\text{D}}(\Mod_R^{\heartsuit, \text{lf}})$.

Let $\alpha: R \rightarrow R$ be the algebra map which sends $x_i$ to $x_{i+1}$ for all $i$. Then restriction of scalars along $\alpha$ provides a colimit preserving left exact functor 
\[
\alpha_*: \Mod_R^{\heartsuit, \text{lf}} \rightarrow \Mod_R^{\heartsuit, \text{lf}}.
\]
The above admits a unique extension to a colimit preserving t-exact functor
\[
\widehat{\text{D}}(\alpha_*): \widehat{\text{D}}(\Mod_R^{\heartsuit, \text{lf}}) \rightarrow \widehat{\text{D}}(\Mod_R^{\heartsuit, \text{lf}}).
\]
Let $F: \QCoh(\B\GG_{a,k}) \rightarrow \QCoh(\B\GG_{a,k})$ be the induced functor. We first show that $F$ admits a colimit preserving right adjoint. To see this, observe that extension of scalars along $\alpha$ provides a functor
\[
\alpha^*: \Mod_R^{\heartsuit, \text{lf}} \rightarrow \Mod_R^{\heartsuit, \text{lf}}
\]
which is adjoint to $\alpha_*$ on both sides. It follows that the corresponding derived functor
\[
\widehat{\text{D}}(\alpha^*): \widehat{\text{D}}(\Mod_R^{\heartsuit, \text{lf}}) \rightarrow \widehat{\text{D}}(\Mod_R^{\heartsuit, \text{lf}}).
\]
is adjoint to $\widehat{\text{D}}(\alpha_*)$ on both sides. The corresponding endofunctor of $\QCoh(\B\GG_{a,k})$ is adjoint to $F$ on both sides, and the claim follows.

It remains to show that the colimit of the diagram from the lemma is equivalent to $\Mod_k$. This is the same as the colimit of the directed system
\[
\widehat{\text{D}}(\Mod_R^{\heartsuit, \text{lf}}) \xrightarrow{\widehat{\text{D}}(\alpha_*)} \widehat{\text{D}}(\Mod_R^{\heartsuit, \text{lf}})  \xrightarrow{\widehat{\text{D}}(\alpha_*)} \widehat{\text{D}}(\Mod_R^{\heartsuit, \text{lf}})  \xrightarrow{\widehat{\text{D}}(\alpha_*)} \ldots .
\]
Applying lemma \ref{lemma t structure colimit} we reduce to showing that the colimit of the system
\[
\Mod_R^{\heartsuit, \text{lf}} \xrightarrow{\alpha_*} \Mod_R^{\heartsuit, \text{lf}} \xrightarrow{\alpha_*} \Mod_R^{\heartsuit, \text{lf}} \xrightarrow{\alpha_*} \ldots
\]
is equivalent to $\Mod_k^\heartsuit$. Observe that the category $\Mod_R^{\heartsuit, \text{lf}}$ is compactly generated by those locally finite $R$-modules which are finite dimensional over $k$ (indeed, this corresponds to the fact that  $\QCoh(\B\GG_{a,k})^\heartsuit$ is compactly generated by finite dimensional representations). Passing to compact objects we reduce to showing that the colimit in $\Cat$ of the system
\[
(\Mod_R^{\heartsuit, \text{lf}})^\omega \xrightarrow{\alpha_*} (\Mod_R^{\heartsuit, \text{lf}})^\omega \xrightarrow{\alpha_*} (\Mod_R^{\heartsuit, \text{lf}})^\omega \xrightarrow{\alpha_*} \ldots
\]
is equivalent to $(\Mod_k^\heartsuit)^\omega$. We may enlarge the above diagram to a commutative diagram
\[
\begin{tikzcd}
(\Mod_R^{\heartsuit, \text{lf}})^\omega \arrow{r}{\alpha_*} & (\Mod_R^{\heartsuit, \text{lf}})^\omega \arrow{r}{\alpha_*} &  (\Mod_R^{\heartsuit, \text{lf}})^\omega \arrow{r}{\alpha_*} & \ldots  \\
(\Mod_k^\heartsuit)^\omega \arrow{r}{\id} \arrow{u}{} & (\Mod_k^\heartsuit)^\omega \arrow{r}{\id} \arrow{u}{} & (\Mod_k^\heartsuit)^\omega \arrow{r}{\id} \arrow{u}{} & \ldots
\end{tikzcd}
\]
where the vertical arrows are the functors of restriction of scalars along the morphism of $k$-algebras $\theta: R \rightarrow k$ sending $x_i$ to $0$ for all $i$. Since the vertical arrows are fully faithful, passing to the colimit of the rows we obtain a fully faithful functor $(\Mod_k^{\heartsuit})^\omega \rightarrow (\Mod_R^{\heartsuit, \text{lf}})^\omega $. To finish the proof it will suffice to show that this functor is surjective. This amounts to the assertion that if $M$ is an object of $(\Mod_R^{\heartsuit, \text{lf}})^\omega$, then for $n$ sufficiently large the $R$-module $\alpha_*^n(M)$ is obtained by restriction of scalars along $\theta$. Indeed, since $M$ is finite dimensional over $k$ and locally finite as an $R$-module, there exists $n$ such that $x_i$ acts by zero on $M$ for $i \geq n$. Then the action of $x_i$ on $\alpha^n_*(M)$ is zero for all $i$, which implies the desired assertion. 
\end{proof}

\begin{lemma}\label{lemma exists endofunctor derived}
Let $k$ be a field of positive characteristic. Then there exists an endofunctor $F: \D(\QCoh(\B \GG_{a,k})^\heartsuit) \rightarrow \D(\QCoh(\B\GG_{a,k})^\heartsuit)$, with a colimit preserving right adjoint, and with the property that the colimit  of the directed system
\[
\D(\QCoh(\B\GG_{a,k})^\heartsuit) \xrightarrow{F} \D(\QCoh(\B\GG_{a,k})^\heartsuit) \xrightarrow{F} \D(\QCoh(\B\GG_{a,k})^\heartsuit) \xrightarrow{F} \ldots 
\]
is equivalent to $\Mod_k$.
\end{lemma}
\begin{proof}
Completely analogous to the proof of lemma \ref{lemma exists endofunctor}.
\end{proof}

\begin{proof}[Proof of theorem \ref{theorem non dualizability bga}]
We give the proof in the case of $\QCoh(\B\GG_{a,k})$; the proof in the case of $\D(\QCoh(\B\GG_{a, k})^\heartsuit)$ is analogous. Assume for the sake of contradiction that $\QCoh(\B\GG_{a,k})$ is dualizable. Combining lemma \ref{lemma exists endofunctor} with \cite{Fully} theorem 3.3.1 and proposition 3.2.9 we see that the colimit in $\Cat$ of the directed system
\[
\QCoh(\B\GG_{a,k})^\omega \xrightarrow{F} \QCoh(\B\GG_{a,k})^\omega \xrightarrow{F} \QCoh(\B\GG_{a,k})^\omega \xrightarrow{F} \ldots 
\]
is given by $\Mod_k^\omega$. In particular, since this colimit is non-zero we deduce the existence of a non-zero compact object in $\QCoh(\B\GG_{a,k})$. Applying \cite{HallNeemanRydh} proposition 3.1 we reach the desired contradiction.
\end{proof}


\subsection{Ind-coherent sheaves} \label{subsection indcoh}

Our goal in this section is to prove the following:

\begin{theorem}\label{theorem indcoh}
Let $X$ be a Noetherian algebraic stack with affine stabilizers.
\begin{enumerate}[\normalfont \indent (a)]
\item Assume that $X$ is well stabilized and has separated diagonal, or that $X$ is very well stabilized. Then $\IndCoh(X)$ is compactly generated.
\item If $\IndCoh(X)$ is dualizable and $\Ocal_X$ is truncated then $X$ is well stabilized.
\end{enumerate}
\end{theorem}

\begin{remark}\label{remark compact is coherent}
Let $X$ be a Noetherian algebraic stack. Then every compact object of $\IndCoh(X)$ is coherent: this follows from the fact that if $p:U \rightarrow X$ is an affine cover then $p^*: \IndCoh(X) \rightarrow \IndCoh(U)$ admits a colimit preserving right adjoint. Note that in general coherent sheaves are not necessarily compact in $\IndCoh(X)$.
\end{remark}

\begin{corollary}\label{coro derived category of heart}
Let $X$ be a classical Noetherian algebraic stack with affine diagonal. Then $\QCoh(X)$ is the derived $\infty$-category of its heart if and only if $X$ is well stabilized.
\end{corollary}
\begin{proof}
The only if direction follows from \cite{HallNeemanRydh} theorem 1.3. Suppose now that $X$ is well stabilized.  To show that $\QCoh(X)$ is the derived $\infty$-category of its heart it suffices to prove that every connective object $\mathcal{F}$ receives a surjection from an object in $\QCoh(X)^\heartsuit$.  Let $\Psi_X: \IndCoh(X) \rightarrow \QCoh(X)$ be the projection, and write $\mathcal{F} = \Psi_X(\mathcal{G})$ for some connective ind-coherent sheaf $\mathcal{G}$.  By theorem \ref{theorem indcoh} we may write $\mathcal{G}$ as the filtered colimit of a diagram of compact ind-coherent sheaves $\mathcal{G}_\alpha$. Each of these is a coherent sheaf by remark \ref{remark compact is coherent}, and in particular it is bounded. We now have that $\mathcal{F}$ is the filtered colimit of the truncated connective quasi-coherent sheaves $\Psi_X(\tau_{\geq 0}\mathcal{G}_\alpha)$, and in particular we have a surjection 
\[
\bigoplus_\alpha \Psi_X(\tau_{\geq 0}\mathcal{G}_\alpha) \rightarrow \mathcal{F}.
\]
It now suffices to show that each truncated connective quasi-coherent sheaf on $X$ receives a surjection from an object in the heart, which follows from \cite{SAG} proposition 10.4.6.6.
\end{proof}

 \begin{corollary}
 Let $X$ be a Noetherian algebraic stack with affine diagonal. Then the following are equivalent:
\begin{enumerate}[\normalfont\indent (i)]
\item Products in $\QCoh(X)$ have bounded cohomological amplitude.
\item Products in $\D(\QCoh(X)^\heartsuit)$ have bounded cohomological amplitude.
\item $X$ is well stabilized.
\end{enumerate} 
 \end{corollary}
 \begin{proof}
 The case of $\QCoh(X)$ was established in corollaries \ref{corollary consequences controlled} and \ref{coro products unbounded}, so it only remains to address the case of $\D(\QCoh(X)^\heartsuit)$. If $X$ is well stabilized then the boundedness of products in $\D(\QCoh(X)^\heartsuit)$ follows from the case of $\QCoh(X)$, by corollary \ref{coro derived category of heart}. Conversely, if products in $\D(\QCoh(X)^\heartsuit)$ have bounded cohomological amplitude then  the t-structure on $\D(\QCoh(X)^\heartsuit)$ is left complete, which implies that $\D(\QCoh(X)^\heartsuit) = \QCoh(X)$. Another application of corollary \ref{coro derived category of heart} now shows that $X$ is well stabilized, as desired. 
 \end{proof}

The remainder of this section is devoted to the proof of theorem \ref{theorem indcoh}.

\begin{lemma}\label{lemma finite surjective indcoh}
Let $f: X \rightarrow Y$ be a morphism between Noetherian algebraic stacks. Suppose that $f$ is affine and $f_*\Ocal_X$ is coherent. If $f$ is surjective and $\IndCoh(X)$ is compactly generated then $\IndCoh(Y)$ is compactly generated.
\end{lemma}
\begin{proof}
Follows from the fact that the functor of pushforward $\IndCoh(X) \rightarrow \IndCoh(Y)$ has a colimit preserving and conservative right adjoint.
\end{proof}

The following lemma is well known, and included only for completeness:

\begin{lemma}\label{lemma indcoh tensored up affine}
Let $S \rightarrow T$ be a smooth morphism of classical Noetherian affine schemes. Then the canonical functor
\[
\IndCoh(S) \otimes_{\QCoh(S)} \QCoh(T) \rightarrow \IndCoh(T)
\]
is an equivalence.
\end{lemma}
\begin{proof}
We first show that the map in the statement is fully faithful. We will do so by showing that the functor $\IndCoh(S)^\cn \rightarrow \IndCoh(T)^\cn$ on connective subcategories is fully faithful. Consider the commutative square of presentable categories
\[
\begin{tikzcd}
\IndCoh(S)^\cn \otimes_{\QCoh(S)^\cn} \QCoh(T)^\cn \arrow{r}{} \arrow{d}{} & \IndCoh(T)^\cn \arrow{d}{} \\
\QCoh(S)^\cn \otimes_{\QCoh(S)^\cn} \QCoh(T)^\cn \arrow{r}{} & \QCoh(T)^\cn
\end{tikzcd}
\]
where the vertical arrows are induced by the projections $\IndCoh(S)^\cn \rightarrow \QCoh(S)^\cn$ and $\IndCoh(T)^\cn \rightarrow \QCoh(T)^\cn$. Note that all the categories in this square are Grothendieck prestable, the vertical arrows induce equivalences of completions, and the bottom horizontal arrow is an equivalence. Consequently, our claim will follow if we show that the top left category is compactly generated by a set of truncated objects, and the top horizontal arrow sends these objects to truncated objects of $\IndCoh(T)^\cn$. Indeed, this is the case for the generators $\Fcal \otimes_{\Ocal(S)} \Ocal(T)$.

It remains to show that the map in the statement is surjective. To see this it will suffice to show that the image of the pullback functor $\IndCoh(S) \rightarrow \IndCoh(T)$ generates $\IndCoh(T)$ under colimits. The smoothness of $S \rightarrow T$ guarantees that the diagonal bimodule of $\Ocal(T)$ is compact, and hence every object $\Gcal$ in $\QCoh(T)$ belongs to the closure under finite colimits, shifts and retracts of the objects $\Ocal(T) \otimes_{\Ocal(S)} \Gcal$. Specializing this assertion to eventually coconnective objects of $\QCoh(T)$ we deduce that the closure under colimits of the image of the functor $\IndCoh(S) \rightarrow \IndCoh(T)$ contains all eventually coconnective objects of $\IndCoh(T)$. The desired assertion now follows from the fact that $\IndCoh(T)$ is generated under colimits by its eventually coconnective objects.
\end{proof}

\begin{lemma}\label{lemma tensor to open}
Let $X$ be a classical Noetherian algebraic stack, and let $Y$ be a quasi-compact open substack of $X$. Then the canonical functor
\[
\IndCoh(X) \otimes_{\QCoh(X)} \QCoh(Y) \rightarrow \IndCoh(Y)
\]
is an equivalence.
\end{lemma}
\begin{proof}
Let $p: U \rightarrow X$ be an affine cover, and let $U_\bullet$ be the \v{C}ech nerve of $p$. Our goal is to show that the canonical functor
\[
(\Tot \IndCoh(U_\bullet) ) \otimes_{\QCoh(X)} \QCoh(Y) \rightarrow \Tot(\IndCoh(Y \times_X U_\bullet))
\]
is an equivalence. The left hand side is equivalent to
\[
\Tot (\IndCoh(U_\bullet) \otimes_{\QCoh(X)} \QCoh(Y) ) = \Tot (\IndCoh(U_\bullet) \otimes_{\QCoh(U_\bullet)} \QCoh(Y \times_X U_\bullet))
\]
so the lemma will follow if we show that the canonical map
\[
\IndCoh(U_\bullet) \otimes_{\QCoh(U_\bullet)} \QCoh(Y \times_X U_\bullet) \rightarrow \IndCoh(Y \times_X U_\bullet)
\]
is an equivalence. Replacing $X$ by $U_n$ we may now assume that $X$ is an algebraic space. Repeating this argument again we may reduce to the case when $X$ is a quasi-affine scheme, and then once again reduces us to the case when $X$ is affine. In this case the desired assertion follows from an application of lemma \ref{lemma indcoh tensored up affine}.
\end{proof}

\begin{lemma}\label{lemma indcoh tensored up}
Let $S$ be a classical Noetherian affine scheme and let $X$ be a smooth classical algebraic stack over $S$. Then the canonical functor
\[
\IndCoh(S) \otimes_{\QCoh(S)} \QCoh(X) \rightarrow \IndCoh(X)
\]
is an equivalence.
\end{lemma}
\begin{proof}
Let $p: U \rightarrow X$ be a smooth affine cover, and let $U_\bullet$ be the \v{C}ech nerve of $p$. Then the morphism in the statement is the totalization of the morphism
\[
\IndCoh(S) \otimes_{\QCoh(S)} \QCoh(U_\bullet) \rightarrow \IndCoh(U_\bullet).
\]
Replacing $X$ by $U_n$ we may now assume that $X$ is an algebraic space. Repeating this argument again we may reduce to the case when $X$ is a quasi-affine scheme, and then once again reduces us to the case when $X$ is affine, which holds by lemma \ref{lemma indcoh tensored up affine}.
\end{proof}

\begin{lemma}\label{lemma indcoh locally on BG}
Let $S$ be an integral Noetherian affine scheme and let $G$ be an algebraic group over $S$ with affine fibers. Assume that $G$ is good. Then there exists a dense open subscheme $U$ of $S$ with the property that $\IndCoh(\B G_U)$ is compactly generated. 
\end{lemma}
\begin{proof}
We apply proposition \ref{proposition structure G}, so that one of conditions (a) or (b) applies.  Suppose first that condition (a) applies. Then we have a quasi-affine map $f: \B G_U \rightarrow \B \GL_{n, \QQ}$. Since the functor $f^*: \QCoh(\B \GL_{n,\QQ}) \rightarrow \QCoh(\B G_U)$ admits a colimit preserving and conservative right adjoint and $\QCoh(\B \GL_{n,\QQ})$ is compactly generated, we deduce that $\QCoh(\B G_U)$ is compactly generated. Since algebraic groups in characteristic zero are smooth we see that the projection $\B G_U \rightarrow U$ is smooth. We now have  $\IndCoh(\B G_U) = \IndCoh(U) \otimes_{\QCoh(U)} \QCoh(\B G_U)$ by lemma \ref{lemma indcoh tensored up} and hence $\IndCoh(\B G_U)$ is also compactly generated, as desired.

Assume now that condition (b) applies. Then the projection $\B G_{U'} \rightarrow \B G_U$ is a finite flat surjection, so by lemma \ref{lemma finite surjective indcoh} it suffices to show that $\IndCoh(\B G_{U'})$ is compactly generated. Fix a subgroup inclusion $\mathbb{G}_{m,U'}^n \rightarrow G_{U'}$ whose quotient is finite flat. Then the induced map $\B \mathbb{G}_{m, U'}^n \rightarrow \B G_{U'}$ is a finite flat surjection. Another application of lemma \ref{lemma finite surjective indcoh} reduces the problem to showing that $\IndCoh(\B\mathbb{G}_{m, U'}^n)$ is compactly generated. The projection $\B \GG_{m, U'}^n \rightarrow U'$ is smooth, so applying lemma \ref{lemma indcoh tensored up} we deduce that $\IndCoh(\B \GG_{m, U'}^n) = \IndCoh(U') \otimes_{\QCoh(U')} \QCoh(\B \GG_{m, U'}^n)$. The desired claim now follows from the fact that $\QCoh(\B \GG_{m, U'}^n) = \QCoh(U')^{\ZZ^n}$ is compactly generated.
\end{proof}

\begin{lemma}\label{lemma indcoh locally}
Let $X$ be a classical Noetherian algebraic stack with affine stabilizers. If $X$ is well stabilized then there exists a dense open substack of $X$ whose category of ind-coherent sheaves is compactly generated.
\end{lemma}
\begin{proof}
By lemma \ref{lemma finite surjective indcoh} it suffices to consider the case when $X$ is reduced. By \cite{stacks-project} proposition 06RC we may, after replacing $X$ by a dense open substack, assume that $X$ is a gerbe over a classical algebraic space. Applying \cite{stacks-project} proposition 06NH we may further reduce to the case when $X$ is a gerbe over a (classical) affine scheme $S$. Changing base to a dense open subset of $S$ we may assume that the connected components of $S$ are irreducible. It suffices then to prove the lemma for each irreducible component of $S$, so that we may assume that $S$ is irreducible. Applying lemma \ref{lemma finite surjective indcoh} we may further assume that $S$ is integral. 

Let $p: S' \rightarrow S$ be an affine cover with the property that $X$ has a section over $S'$. Changing base to a dense open subscheme of $S$ we may assume that there exists a finite flat surjection $S'' \rightarrow S$ refining $S' \rightarrow S$, whose base change to the generic point of $S$ is the spectrum of a field. Replacing $S'$ by $S''$ we may now assume that $p$ is finite flat and that the generic fiber of $p$ is the spectrum of a field. Changing base to a dense open subscheme of $S$ we may now assume that $S'$ is integral.

Let $X' = S' \times_S X'$. Then $X'$ is of the form $\B G$ for some algebraic group $G$ over $S'$ with affine fibers. By proposition \ref{proposition algebraic space over well stabilized} we have that $X'$ is well stabilized. It now follows from proposition \ref{proposition BG well stabilized iff G good} that $G$ is good. Consequently, applying lemma \ref{lemma indcoh locally on BG} we may find a dense open subscheme $U'$ of $S'$ such that $\IndCoh(U' \times_S X)$ is compactly generated.

Since $p$ is open we have that $p(U')$ contains the generic point of $S$, and hence the generic fiber of $p|_{U'}$ is equal to the generic fiber of $p$. Consequently, we may find an affine dense open subscheme $U$ of $S$ with the property that $p^{-1}(U)$ is contained in $U'$. Applying lemma \ref{lemma tensor to open} we deduce that the pullback functor  $\IndCoh(U' \times_S X) \rightarrow \IndCoh(p^{-1}(U) \times_S X)$ admits a fully faithful and colimit preserving right adjoint, and hence $\IndCoh(p^{-1}(U) \times_S X)$ is compactly generated. We now have that $\IndCoh(U \times_S X)$ is compactly generated by lemma \ref{lemma finite surjective indcoh}.
\end{proof}

\begin{proof}[Proof of theorem \ref{theorem indcoh}]
We first prove (b). Since $\Ocal_X$ is truncated the projection $\Psi_X: \IndCoh(X) \rightarrow \QCoh(X)$ admits a fully faithful left adjoint, and consequently $\QCoh(X)$ is a retract of $\IndCoh(X)$ in $\Pr_\st$. It follows that $\QCoh(X)$ is also dualizable, and hence $X$ is well stabilized by corollary \ref{corollary non dualizability poorly stabilized}.

It remains to prove (a). Using lemma \ref{lemma finite surjective indcoh} we may assume that $X$ is reduced. By lemma \ref{lemma indcoh tensored up affine} we have
\[
\IndCoh(X) = \Gamma(X, \IndCoh_X)
\]
where $\IndCoh_X$ is the object of $\twoQCoh(X)$ which assigns the category $\IndCoh(U)$ to every smooth morphism $p: U \rightarrow X$ with $U$ affine. Applying corollary \ref{corollary consequences controlled} we now see that $\IndCoh(X)$ is dualizable.

By Noetherian induction we may assume that $\IndCoh(Z)$ is compactly generated for every proper reduced closed substack $Z$ of $X$. The case when $X$ is empty is clear, so assume that $X$ is nonempty. By lemma \ref{lemma indcoh locally} there exists a dense open substack $U$ of $X$ such that $\IndCoh(U)$ is compactly generated. Let $Z$ be the reduced complement of $U$ and consider the short exact sequence of stable categories
\[
0 \rightarrow \IndCoh(X)_Z \rightarrow \IndCoh(X) \rightarrow \IndCoh(U) \rightarrow 0.
\]
By our inductive hypothesis we have that $\IndCoh(Z)$ is compactly generated. Combined with the fact that the pushforward functor $\IndCoh(Z) \rightarrow \IndCoh(X)_{Z}$ admits a colimit preserving and conservative right adjoint we deduce that $\IndCoh(X)_Z$ is compactly generated as well. The fact that $\IndCoh(X)$ is compactly generated now follows from an application of \cite{Efimov} proposition 3.3 (since we already know it to be dualizable).
\end{proof}


\ifx\inmain\undefined
\bibliographystyle{myamsalpha2}
\bibliography{References}
\fi


\appendix

\section{Classifying stacks of infinite type affine group schemes}\label{appendix infinite type}

The goal of this appendix is to study the question of dualizability of $\QCoh(\B G)$ for affine group schemes beyond the finite type setting. We focus on the case of fields of characteristic zero, where the difference between finite and infinite type is stark: while $\QCoh(\B G)$ is always dualizable (and, in fact, compactly generated) for $G$ an affine algebraic group in characteristic zero, this is often not the case in the infinite type setting. 

We may formulate our main result as follows:

\begin{theorem}\label{theorem dualizable infinite}
Let $G$ be a affine group scheme over a field $k$ of characteristic zero. The following are equivalent:
\begin{enumerate}[\normalfont \indent (1)]
\item $\QCoh(\B G)$ is dualizable.
\item $\QCoh(\B G)$ is compactly generated.
\item $\Gamma(\B G, -)$ has bounded cohomological amplitude.
\item $\Ocal_{\B G}$ is a compact object of $\QCoh(\B G)$.
\end{enumerate}
\end{theorem}

Before giving the proof of theorem \ref{theorem dualizable infinite} we discuss some consequences.

\begin{corollary}\label{corollary BGainfty}
Let $k$ be a field and let $I$ be an infinite set. Then $\QCoh(\B \GG_{a, k}^I)$ is not dualizable.
\end{corollary}
\begin{proof}
In characteristic $p$ the result follows from a combination of theorem \ref{theorem non dualizability bga} and \cite{Efimov} proposition 1.57, using the fact that given an element in $I$ the corresponding functor of pullback along $\B \GG_{a, k} \rightarrow \B \GG_{a, k}^I$ admits a conservative and colimit preserving right adjoint. Meanwhile, in characteristic zero it follows from theorem \ref{theorem dualizable infinite} together with the fact that 
\[
\Gamma(\B \GG_{a, k}^I, \Ocal_{\B \GG_{a, k}^I}) = \bigotimes_I \Gamma(\B \GG_{a, k}, \Ocal_{\GG_{a,k}})
\]
is cohomologically unbounded when $I$ is infinite.
\end{proof}

\begin{corollary}\label{corollary pass to subgroup}
Let $G$ be an affine group scheme over a field $k$ of characteristic zero. If $\QCoh(\B G)$ is dualizable then $\QCoh(\B H)$ is dualizable for every closed subgroup $H$ of $G$. 
\end{corollary}
\begin{proof}
By theorem \ref{theorem dualizable infinite} it suffices to show that if $\Ocal_{\B G}$ is compact then $\Ocal_{\B H}$ is compact. This follows from the fact that the functor of pullback along $\B H \rightarrow \B G$ has a colimit preserving right adjoint.
\end{proof}

\begin{corollary}\label{coro if you contain product}
Let $G$ be an affine group scheme over a field $k$ of characteristic zero. If $G$ contains a closed subgroup isomorphic to an infinite product of copies of $\GG_{a, k}$ then $\QCoh(\B G)$ is not dualizable.
\end{corollary}
\begin{proof}
Combine corollaries \ref{corollary BGainfty} and \ref{corollary pass to subgroup}
\end{proof}

\begin{example}\label{example paths}
Let $k$ be a field and let $G$ be an affine algebraic group over $k$. For each $i\geq 1$ let $G[t]/t^i$ be the affine algebraic group over $k$ whose functor of points sends each $R$-algebra to $G(R[t]/t^i)$, and let $G[[t]]$ be the limit of the diagram
\[
G[t]/t^1 \leftarrow G[t]/t^2 \leftarrow G[t]/t^3 \leftarrow \ldots .
\] 
For each $i$ we have a fiber sequence of groups
\[
G[t]/t^{i+1} \rightarrow G[t]/t^i \rightarrow \B \mathfrak{g},
\]
and hence the kernel of the map $G[t]/t^{i+1} \rightarrow G[t]/t^i$ is given by $\mathfrak{g}$. In particular, the kernels of the maps $G \rightarrow G[t]/t^i$ are pro-unipotent; these are called the congruence subgroups of $G[[t]]$. 

We claim that if $G$ is positive dimensional then $\QCoh(\B G[[t]])$ is not dualizable. Let $K$ be an algebraic closure of $k$. Then the functor of pullback $\QCoh(\B G[[t]]) \rightarrow \QCoh(\B G_K[[t]])$ admits a colimit preserving and conservative right adjoint, so using \cite{Efimov} proposition 1.57 we may reduce to the case when $k$ is algebraically closed. In this case, $G$ contains a copy of $\GG_{a, k}$ or $\GG_{m, k}$. Then $G[[t]]$ contains a copy of $\GG_{a,k}[[t]]$ or $\GG_{m, k}[[t]]$, whose congruence  subgroups are isomorphic to infinite product of copies of $\GG_{a, k}$. The fact that $\QCoh(\B G[[t]])$ is not dualizable is now a consequence of corollary \ref{coro if you contain product}.
\end{example}

We now turn to the proof of theorem \ref{theorem dualizable infinite}.

\begin{lemma}\label{lemma short exact sequence bounds}
Let
\[
1 \rightarrow N \rightarrow G \rightarrow H \rightarrow 1
\]
be a short exact sequence of affine group schemes over a field $k$. If $\Gamma(\B N, -)$ and $\Gamma(\B H, -)$ have cohomological amplitude bounded by $d \geq 0$ then $\Gamma(\B G, -)$ has cohomological amplitude bounded by $2d$.
\end{lemma}
\begin{proof}
The functor $\Gamma(\B G, -)$ factors as the composition of the functor $\Gamma(\B H, -)$ with the functor of pushforward along the projection $p: \B G \rightarrow \B H$, so it suffices to show that the latter has cohomological amplitude bounded by $d$. By lemma \ref{lemma check bound on heart} we may reduce to showing that $p_* \Fcal$ is $(-d)$-connective for every object $\Fcal$ in $\QCoh(\B G)^\heartsuit$. Indeed, the fiber of $p_*\Fcal$ is given by $\Gamma(\B N, \Fcal')$, where $\Fcal'$ is the pullback of $\Fcal$ to $\B N$, so our claim follows from the given bound on $\Gamma(\B N, -)$.
\end{proof}

\begin{lemma}\label{lemma colimit on heart}
Let $G_\alpha$ be a cofiltered diagram of affine group schemes over a field $k$, with limit $G$. Then the canonical functor
\[
\colim \QCoh(\B G_\alpha)^\heartsuit \rightarrow \QCoh(\B G)^\heartsuit
\]
is an equivalence.
\end{lemma}
\begin{proof}
It suffices to show that the functor $G \mapsto \QCoh(\B G)^\heartsuit$ on the category of affine group schemes is Kan extended from its restriction to the full subcategory on the affine algebraic groups. To see this it suffices to prove that if $G_\alpha$ is a cofiltered diagram of affine algebraic groups and quotient maps with limit $G$ then the canonical functor
\[
\colim \QCoh(\B G_\alpha)^\heartsuit \rightarrow \QCoh(\B G)^\heartsuit
\]
is an equivalence. All of the categories above are compactly generated and the functors are compact, so we reduce to showing that the comparison map
\[
(\colim \QCoh(\B G_\alpha)^\heartsuit)^\omega \rightarrow (\QCoh(\B G)^\heartsuit)^\omega
\]
is an equivalence, where the colimit is taken in $\Cat$. Since the functors of restriction of representations along quotients of group schemes are fully faithful, we see that the above comparison map is fully faithful. To finish the proof of the lemma it remains to show that any finite dimensional representation of $G$ is restricted from $G_\alpha$ for some $\alpha$. Indeed, this follows from the fact that every finite dimensional representation of $G$ is restricted from the defining representation of the algebraic group $\GL_{n, k}$ for some $n$.
\end{proof}

\begin{proof}[Proof of theorem \ref{theorem dualizable infinite}]
We show first the equivalence between (3) and (4). Suppose first that (3) holds, so that $\Gamma(\B G, -)$ has cohomological amplitude bounded by $d$ for some $d \geq 0$. Our goal is to show that for every filtered system $\Fcal_\alpha$ of objects of $\QCoh(\B G)$ we have
\[
\tau_{\geq 0} \Hom_{\QCoh(\B G)}(\Ocal_{\B G}, \colim \Fcal_\alpha) = \colim \tau_{\geq 0} \Hom_{\QCoh(\B G)}(\Ocal_{\B G},\Fcal_\alpha).
\]
Passing to the connective covers of $\Fcal_\alpha$ does not change either side, so without loss of generality we assume that $\Fcal_\alpha$ is connective for all $\alpha$. To prove the above isomorphism it suffices to show that for every $m \geq 0$ the induced map
\[
\tau_{\leq m} \tau_{\geq 0} \Hom_{\QCoh(\B G)}(\Ocal_{\B G}, \colim \Fcal_\alpha) = \tau_{\leq m} \colim \tau_{\geq 0} \Hom_{\QCoh(\B G)}(\Ocal_{\B G},\Fcal_\alpha)
\]
is an isomorphism. Using our bound on $\Gamma(\B G, -)$ we see that both sides do not change if we replace $\Fcal_\alpha$ by $\tau_{\leq m+d} \Fcal_\alpha$. We have thus reduced to showing the case when $\Fcal_\alpha$ is $m + d$-truncated, which follows from the fact that $\Ocal_{\B G}$ is an almost compact object of $\QCoh(\B G)$.

Suppose now that (4) holds. Then for every sequence of connective objects $\Fcal_i$ in $\QCoh(X)$ we have
\[
\textstyle{\prod} \Gamma(\B G, \Fcal_i)[i] = \Gamma(\B G, \textstyle{\prod}\Fcal_i[i]) = \Gamma(\B G, \bigoplus\Fcal_i[i]) = \bigoplus \Gamma(\B G, \Fcal_i)[i].
\]
Passing to $\pi_0$ we deduce that every sequence of maps $\Ocal_{\B G} \rightarrow \Fcal_i[i]$ is eventually zero. Since the sequence $\Fcal_i$ was arbitrary, we deduce that $\Gamma(\B G, -)$ has bounded cohomological amplitude.

We claim now that (3) and (4) imply (2). Suppose that $\Gamma(\B G, -)$ has cohomological amplitude bounded by $d \geq 0$. Note that if $\Vcal$ is a finite dimensional representation of $G$ we have
\[
\Hom_{\QCoh(\B G)}(\Vcal, -) = \Hom_{\QCoh(\B G)}(\Ocal_{\B G}, \Vcal^\vee \otimes -)
\]
which implies that $\Vcal$ is a compact object of $\QCoh(X)$ and the functor $\Hom_{\QCoh(\B G)}(\Vcal, -)$ has cohomological amplitude bounded by $d$. Assertion (2) will then follow if we show that these objects generate $\QCoh(\bg)$ under colimits and shifts. To see this it is enough to show that $\QCoh(\bg)$ is the derived category of its heart, which will follow if we show that every connective object $\Fcal$ receives an effective epimorphism from an object in the heart. Since $\QCoh(\bg)$ is the completed derived category of its heart we may find an object $\Gcal$ in $\QCoh(\bg)^\heartsuit$ and an effective epimorphism $f: \Gcal \rightarrow \tau_{\leq d}\Fcal$. Without loss of generality we may assume that $\Gcal$ is a direct sum of finite dimensional representations of $G$. Then $\Hom_{\QCoh(\bg)}(\Gcal, -)$ has cohomological amplitude bounded by $d$, and consequently the map $f$ lifts to an effective epimorphism $\Gcal \rightarrow \Fcal$. This finishes the proof of (2).

Condition (2) clearly implies condition (1). We will finish the proof by showing that (1) implies (3).  Write $G$ as the limit of a cofiltered diagram $G_\alpha$ of affine algebraic groups and quotient maps, and for each $\alpha$ let $N_\alpha$ be the kernel of the projection $G \rightarrow G_\alpha$. Then we have that $\lim N_\alpha$ is the trivial group, and consequently by lemma \ref{lemma colimit on heart} we have an equivalence
\[
\colim \QCoh(\B N_\alpha)^\heartsuit = \Mod_k.
\]
Applying lemma \ref{lemma t structure colimit} we have that the above upgrades to an equivalence
\[
\colim \QCoh(\B N_\alpha) = \Mod_k.
\]
Since the functor of pullback $\QCoh(\B G) \rightarrow \QCoh(\B N_\alpha)$ admits a conservative and colimit preserving right adjoint, we see that $\QCoh(\B N_\alpha)$ is dualizable for all $\alpha$ by \cite{Efimov} proposition 1.57. Applying \cite{Fully} theorem 3.3.1 and proposition 3.2.9 we see that there exists $\alpha$ and a compact object $\Fcal$ in $\QCoh(\B N_\alpha)$ whose image in $\Mod_k$ is given by $k$. In particular, $\Fcal$ is a compact object of the heart of $\QCoh(\B N_\alpha)$. Since $\Ocal_{\B N_\alpha}$ is also a compact object of $\QCoh(\B N_\alpha)^\heartsuit$ whose image in $\Mod_{k}^\heartsuit$ is given by $k$ we have that there exists an index $\beta \rightarrow \alpha$ such that the pullback of $\Fcal$ to $\B N_\beta$ is isomorphic to $\Ocal_{\B N_\beta}$. In other words, we have proven that the structure sheaf of $\B N_\beta$ is compact. Using that (3) and (4) are equivalent for $N_\beta$ we deduce that $\Gamma(\B N_\beta, -)$ has bounded cohomological amplitude. We now conclude that (3) holds, by virtue of lemma \ref{lemma short exact sequence bounds}.
\end{proof}


\ifx\inmain\undefined
\bibliographystyle{myamsalpha2}
\bibliography{References}
\fi


\bibliographystyle{myamsalpha2}
\bibliography{References}
 
\end{document}